\documentclass[11pt,twoside]{article}
\usepackage{latexsym}
\usepackage{amssymb,amsbsy,amsmath,amsfonts,amssymb,amscd}
\usepackage{graphicx,color}
\usepackage{xcolor}
\usepackage{float,url}
\usepackage{cancel}
\usepackage[colorlinks,linkcolor=red,citecolor=blue]{hyperref}
\usepackage{ulem}
\usepackage[left,modulo]{lineno}
\usepackage{comment}
\newcommand  \ind[1]  {   {1\hspace{-1.2mm}{\rm I}}_{\{#1\} }    }

\newcommand{\commentout}[1]{}
\newcommand {\al} {\alpha}
\newcommand {\eps}  {\varepsilon}

\newcommand {\lb} {\lambda}
\newcommand {\Chi} {{\bf \raise 2pt \hbox{$\chi$}} }

\newcommand {\dv}  { {\rm div} }
\newcommand{\dis}{\displaystyle}

\newcommand{\red}{\textcolor{black}}
\newcommand{\blue}{\textcolor{black}}
\newcommand{\beq}{\begin{equation}}
\newcommand{\eeq}{\end{equation}}
\newcommand{\bea} {\begin{array}{rl}}
\newcommand{\eea} {\end{array}}
\newcommand{\bepa}{\left\{ \begin{array}{l}}
\newcommand{\eepa} {\end{array}\right.}
\newtheorem{theorem}{Theorem}
\newtheorem{lemma}[theorem]{Lemma}

\newtheorem{remark}[theorem]{Remark}
\newtheorem{proposition}[theorem]{Proposition}
\newtheorem{corollary}[theorem]{Corollary}
\newcommand{\qed}{{ \hfill
                       {\unskip\kern 6pt\penalty 500 \raise -2pt\hbox{\vrule\vbox to 6pt{\hrule width 6pt
                       \vfill\hrule}\vrule} \par}   }}
\newenvironment{proof}{
  \par\noindent\textit{Proof.}\
}{ \par
\hfill$\square$\par \medskip
}

\title{Multiscale analysis of a kinetic equation for mechanotaxis}

\author{
Beno\^ \i t Perthame\thanks{Sorbonne Universit{\'e}, CNRS, Universit\'{e} de Paris Cit\'e, Inria, Laboratoire Jacques-Louis Lions,  F-75005 Paris. 
Email: Benoit.Perthame@sorbonne-universite.fr}
\and
Francesco Salvarani\thanks{De Vinci Higher Education, De Vinci Research Center, Paris, France \& Dipartimento di Matematica ``F.
Casorati'', Universit\`a degli Studi di Pavia, Italy.
Email: francesco.salvarani@unipv.it}
\and
Shugo Yasuda\thanks{Graduate School of Information Science, University of Hyogo, Kobe 650-0047, Japan. Email: yasuda@gsis.u-hyogo.ac.jp}
}
\date{\today}

\begin{document}
\maketitle
\pagestyle{plain}
\pagenumbering{arabic}

\begin{abstract} 
We present a new kinetic equation for cell migration driven by mechanical interactions with the substrate, an effect not previously captured in kinetic models,
and essential for explaining observed collective behaviors such as those in bacterial colonies. The model introduces an acceleration term that accounts for the dynamics of motile cells undergoing mechanotaxis, where extracellular signals modulate the forces arising from cell–substrate interactions. From this formulation, we derive a family of macroscopic limit equations and analyze their principal properties. In particular, we examine linear stability and pattern formation ability through theoretical analysis, supported by numerical simulations.
\end{abstract} 
\vskip .7cm

\noindent{\makebox[1in]\hrulefill}\newline

%
\section{Introduction}
\label{sec:intro}

Cell migration guided by mechanical signals is widely 
reported and can play a central role in directing cell motion across a range of biological systems~\cite{LO2000144}.
This phenomenon, commonly referred to as durotaxis, and more generally as mechanotaxis, contributes to key physiological and pathological processes such as wound healing, tissue regeneration, and collective cancer cell invasion~\cite{Malik2020,doi:10.1126/science.aaf7119}.
Mechanical interactions between cells and their substrates have also been shown to play important roles in collective bacterial behaviors, including swarming, pattern formation, and biofilm formation, where bacteria actively sense and modify their mechanical environment to regulate collective dynamics~\cite{Partridge2013, Persat2015, Yang2017}. 
For instance, cyanobacteria can secrete extracellular substances that alter substrate mechanical properties, leading to effective collective guidance of the population~\cite{Ursell2013}.

Understanding the mechanisms that govern these phenomena is essential not only for theoretical biology but also for practical applications in biomedical and environmental engineering, including the design of biomaterials and the development of regenerative therapies and bioremediation strategies. 

However, the inherent complexity of these processes, which involves multiscale interactions between physical forces, chemical signals, and cellular dynamics, makes it challenging to construct predictive models that are both accurate and computationally feasible.

This work addresses the need to bridge the gap between kinetic models, which describe the behavior of cells at a mesoscopic scale, and macroscopic models, which capture the collective evolution of cell populations. 
In particular, the study focuses on the dominant role of friction and its interplay with stochastic reorientation mechanisms (tumbling), both of which critically influence pattern formation and the stability of spatial configurations. 
The objective is to provide a \red{formal} derivation of asymptotic limits that connect kinetic equations to diffusion-type models, clarifying how different time scales and friction parameters lead to qualitatively distinct behaviors.

The purpose of this research is hence to derive reduced models that retain the descriptive power of the original systems while remaining computationally tractable. Such models are essential for analyzing phenomena such as cell aggregation, the emergence of periodic structures and the conditions that trigger Turing-type instabilities.

\subsection{The kinetic model for mechanotaxis}
\label{sec:kinetic}

More precisely, we study the following kinetic equation:
\begin{equation}
\begin{cases}
\partial_t f + V \Omega \cdot \nabla_x f +\displaystyle \frac 1 m  \partial_V \left[ ({F} - V \mu(S)) f \right] = {\lambda} \mathcal{T}[f], \\[8pt]
f(t,x,\Omega, V=0) = 0. \\[8pt]
\end{cases}
\label{eq:K1}
\end{equation}
Here $x \in \mathbb{R}^d$ ($d\in \mathbb{N}^*$) denotes the spatial position, $V > 0$ the magnitude of the cell velocity, and $\Omega \in \mathbb{S}^{d-1}$ the direction of motion, normalized such that $|\mathbb{S}^{d-1}| = 1$. The distribution function $f\,~:~\,[0,T]\times \mathbb{R}^d\times\mathbb{S}^{d-1}\times \mathbb{R}_+ \to \mathbb{R}_+$ represents the density of cells at time $t$, position $x$, moving in direction $\Omega$ with speed $V$.

This model describes the dynamics of motile cells undergoing mechanotaxis, where the force exerted by the interaction with the substrate is influenced by a signal $S$. 
The transport term $V \Omega \cdot \nabla_x f$ accounts for the spatial movement of cells, while the velocity drift term $\frac 1 m \partial_V[(F - V \mu(S)) f]$, where $m$ represents the typical cell mass,  models acceleration or deceleration due to internal or external stimuli. The term $F - V \mu(S)$ represents the net force acting on the cell in the velocity space, where $F>0$ is a generic propulsion force and $S\mapsto \mu(S) >0$ encodes the influence of the signal on friction. 

The right-hand side contains the \textit{tumbling operator} $\mathcal{T}[f]$, which models random reorientation events. Tumbling changes the direction $\Omega$ but not the speed $V$ and is assumed to occur with rate $\lambda$.
\blue{It takes the general form of an integral operator where the tumbling kernel $K(\Omega,\Omega')$ denotes the probability of a direction jump from $\Omega'$ to $\Omega$ 
\begin{equation}
\mathcal{T}[f] = \int_{\mathbb{S}^{d-1}}K(\Omega,\Omega')f(\Omega')\mathrm{d}\Omega' - f(\Omega), \qquad \int_{\mathbb{S}^{d-1}}K(\Omega,\Omega')\mathrm{d}\Omega=1.
\label{eq:T}
\end{equation}
The simplest form of the tumbling operator assumes that after a reorientation event, the new direction is uniformly distributed.
In the following, for better readability, we consider only the uniform tumbling kernel, i.e., 
$$
K(\Omega,\Omega')=\frac{1}{|\mathbb{S}^{d-1}|}.
$$
}
However, even for more general tumbling kernels having the \red{form $\tilde K(\Omega \cdot\Omega')>0$,  we can obtain the same continuum} limit equations except that the diffusion constant is modified \cite{allaire2018transport, Perthame2004_chemo}. \blue{ This assumption also shows that our approach differs deeply from the usual derivation of the Keller-Segel system, where it is assumed that the tumbling kernel is biased towards higher signals \cite{AltDO88, Chalub2004, Perthame2004_chemo}. Here, the mechanism for obtaining a macroscopic equation only stems from the variations of friction.
}

The boundary condition $f(t,x,\Omega, V=0) = 0$ ensures that cells cannot have zero velocity. The macroscopic cell density is defined by:
$$
\rho(t,x) = \int_{{\mathbb R}_+} \int_{\mathbb{S}^{d-1}} f(t,x,\Omega,V) \, \mathrm{d} \Omega \, dV.
$$

For convenience, we introduce the following integrated quantities:

\begin{equation*}
g(t,x,V) := \int_{\mathbb{S}^{d-1}} f(t,x,\Omega,V) \, \mathrm{d}\Omega, \qquad
p(t,x,\Omega) := \int_{{\mathbb R}_+}  f(t,x,\Omega,V) \, dV.
\end{equation*}

\subsection{Literature review}
Using kinetic equations to describe the run-and-tumble movement goes back to the 80's and 90's, \cite{AltDO88,OthmerStevens97}, and was motivated by the first experimental observations of this movement for bacteria. In the case of eukaryotic cells moving in a tissue, the so-called mesenchymal movement, cells adhere to a fiber network, the extracellular matrix (ECM), and the re-orientation process depends on the fibers' directions. This leads to use more sophisticated tumbling kernels in the kinetic equations, in particular to take into account the degradation and remodeling of the ECM by the cells \cite{HillenM5_05}, and, at the macroscopic level, to derive drift-diffusion equations where the mean drift velocity is determined by the mean orientation of the tissue. Migrating cells can sense their environment, leading to nonlocal kinetic models. Such models were established in~\cite{HPS2007} where a nonlocal macroscopic limit is also derived. 
In \cite{Painter_09}, it is also proposed to make a distinction, in the velocity variable, between orientation and velocity modulus. Another class of models takes into account that cells can use a non-local sensing of the surrounding cell density to decide of the re-orientation directions after tumble, see  \cite{LPrez_20,CLoy_22}. This process leads to high cell concentrations and singularities, see~\cite{LPrez_21,LoyP24} in long times, as it occurs in short times in the seminal Keller-Segel system. Let us also mention that internal states of the cells, deciding of their tumbling rate, has been studied in~\cite{XXTang18,Yasuda2022} with consequences closely related to our present analysis.

In these papers the force exerted by interaction with  the substrate is not taken into account. Here we propose to consider the mechanical force  resulting from the interaction with the substrate. In particular this allows us to take into account the modification of mechanical properties of the substrate by the cells themselves. 

\blue{Our pattern formation results are consistent with the broader chemotaxis literature, where instability of homogeneous steady states and the emergence of structured
stationary patterns have been analyzed in chemotaxis systems. In particular, increasing chemotactic effects can destabilize homogeneous equilibria and generate nonhomogeneous states
such as spikes, stripes, rolls, square or hexagonal patterns, depending on the structure of the model and the parameter regime. See, for instance, \cite{mimuraJJIAM2010,painter1999stripe}. 
}

{\color{black}
\subsection{Scaling estimates and biological interpretation}
The basic kinetic equation (\ref{eq:K1}) is nondimensionalized as
$$
\partial_{\widehat t}\widehat f +\widehat V\Omega\cdot\nabla_{\widehat x}\widehat f
+\frac{1}{\widehat m}\partial_{\widehat V}[(\widehat F-\widehat V\widehat \mu(S))\widehat f]=\frac{1}{\varepsilon}\widehat \lambda{\cal T}[\widehat f],
$$
where the dimensionless parameters $\widehat m$ and $\varepsilon$ are defined by
$$
\widehat m=\frac{m/\mu_\mathrm{c}}{t_\mathrm{c}},\quad
\varepsilon=\frac{{\lambda_\mathrm{c}}^{-1}}{t_\mathrm{c}},
$$
with  $t_\mathrm{c}=L_\mathrm{c}/V_\mathrm{c}$.
Here, the notation ``$ \widehat{\;}$'' denotes nondimensional variables, and the subscript~$\mathrm{c}$ indicates characteristic quantities.
The propulsion force is scaled as $F=F_\mathrm{c}\widehat F$ with $F_\mathrm{c}=V_\mathrm{c}\mu_\mathrm{c}$.
It should be noted that $\widehat{m}$ represents the dimensionless velocity relaxation time, whereas $\varepsilon$ represents the dimensionless run time.

The velocity relaxation time $m/\mu_\mathrm{c}$ depends strongly on how the cell mechanically interacts with its external environment.
At the level of a simple scaling estimate, the densities of the cell and its surrounding biological environment may be regarded as being of the same order.
Therefore, density ratios do not explicitly enter the following estimates, and the dominant dependence is expressed in terms of the cell size $d$ and the effective interaction area $A_\mathrm{eff}$. 
For a cell moving in a fluid, hydrodynamic drag gives $m/\mu_\mathrm{c}\sim d^2/\nu_\mathrm{c}$, where $\nu_\mathrm{c}$ is the kinematic viscosity of the fluid.
For a cell crawling on a thin liquid layer of thickness $h$, as in the gliding motion of cells, a lubrication-type estimate gives $m/\mu_\mathrm{c}\sim d^3/(\nu_\mathrm{c}A_{\rm eff}/h)$.
For substrate-based motion, if the effective friction coefficient scales with the contact area, then $m/\mu_\mathrm{c}\sim d^3/(\zeta A_{\rm eff})$, where $\zeta$ is a density-normalized effective friction parameter with the dimension of velocity.

We now rewrite the dimensionless parameter $\widehat m$ as $\widehat m=\alpha \varepsilon$, where
$$
\alpha=\frac{m\lambda_\mathrm{c}}{\mu_\mathrm{c}}.
$$
The parameter $\alpha$ measures the ratio of the velocity relaxation time $m/\mu_\mathrm{c}$ to the run time $1/\lambda_\mathrm{c}$.

The run time of cells, $\lambda_\mathrm{c}^{-1}$, varies widely depending on the species and their motility state. 
For example, run-and-tumble bacteria such as \textit{Escherichia coli} typically exhibit $\lambda_\mathrm{c}^{-1}\sim~1\,\,[\mathrm{s}]$. 
In contrast, in run-and-tumble-like motions of eukaryotic cells, the run time can be much longer, typically of the order $\lambda_\mathrm{c}^{-1}\sim 10^2\,\,[\mathrm{s}]$, \cite{Uwamichi2023,Zhang2023}.
Thus, for bacteria and eukaryotic cells of microscopic size, say \(d<100\,\mu\mathrm{m}\), the parameter \(\alpha\) is expected to be negligibly small in most biologically relevant situations.
Nevertheless, from the mathematical point of view, it is natural to investigate distinguished scalings of $\alpha$ since they lead to distinct macroscopic limits.

For bacteria moving in fluids, the signal $S$ can be a chemical concentration that modifies the fluid viscosity.
On the other hand, for bacteria crawling on a thin fluid layer, the signal $S$ can represent the thickness of the fluid layer or the presence of a lubricating surfactant, which may be produced by the cells themselves.
}

\subsection{Structure of the paper}
The structure of the article is the following.
After the introduction, the paper analyzes various asymptotic regimes. Several scaling limits are examined: the strong friction limit, which leads to a modified Keller-Segel-type equation; the diffusion limit, which emerges under fast tumbling dynamics, and an intermediate regime that connects these two cases through a scaling parameter. 
For each regime, we derive the corresponding effective equations.

The subsequent sections explore the macroscopic behaviour of the system, analyzing entropy dissipation properties and introducing a condition for the entropy control as well as for the free energy decay.

The study then addresses pattern formation: by linearizing around a steady state, it identifies instability conditions that can give rise to spatially periodic structures, thanks to Turing mechanisms.
We also classify nonlinear patterns theoretically.
Finally, the article presents numerical simulations in one space dimension, which validate the theoretical predictions and illustrate phenomena such as the transition between various patterns depending on friction-diffusion parameters.

The paper concludes with a discussion of the biological and mathematical implications of the results, emphasizing how the proposed multiscale framework can serve as a foundation for future developments in both theoretical analysis and practical applications in tissue modeling and regenerative medicine.

\section{Macroscopic limits}

Two physical quantities permit to rescale the kinetic equation \eqref{eq:K1}, namely the tumbling rate $\lambda$ and the cell mass \red{represented} by $m$. The rigorous derivation of macroscopic equations by diffusion limit is well established \cite{BSS85,Chalub2004} and we adapt it only at a formal level.

We begin with two simple peculiar limits when these parameters vanish separately and turn, in the third subsection, to the general case. 
In this section we consider that $S$ is smooth and possibly a functional of the macroscopic density $\rho$.

\subsection{Friction dominating}

 A possible scaling is just to consider that the cell mass is small, $m \approx 0$,  while the tumbling rate stays of order $1$,  $\lambda =O(1)$. With standard arguments, see \cite{PTV2016}, we obtain
\begin{equation}
f \to p (t,x, \Omega) \delta \left(V= \frac{F}{\mu(S)}\right) \qquad \text{as } m \to 0.
\label{eq:cv}
\end{equation}
Hence $p $ satisfies
\begin{equation} \begin{cases}
\displaystyle
\partial_t p  +\Omega\cdot \nabla_x 
\left[ \frac{F}{\mu(S)} p  \right] = \lambda {\mathcal T}[ p  ] ,
\\[5pt]
\displaystyle
\rho  (t,x) = \dis \int_{{\mathbb S}^{d-1}} p (t,x, \Omega)\, \mathrm{d}\Omega.
\end{cases}
\label{eq:K2}
\end{equation}
This is still a kinetic equation because of the orientation $\Omega$. To obtain a \red{macroscopic} equation, we still need to argue on the tumbling rate.

%
\paragraph{Strong friction, diffusion limit.}

The diffusion limit is obtained by rescaling \eqref{eq:K2} in time  with $\eps= \frac 1 \lb \approx 0$:
\begin{equation} \label{eq:K3}
\begin{cases}
\displaystyle
\eps \partial_t p_\eps +\Omega\cdot \nabla_x 
\left[ \frac{F}{\mu(S)} p_\eps \right] = \frac 1 \eps {\mathcal T}[ p_\eps ] ,
\\[5pt]
\displaystyle
\rho_\eps (t,x) = \dis \int_{{\mathbb S}^{d-1}}   p_\eps(t,x, \Omega) \, \mathrm{d}\Omega.
\end{cases}
\end{equation}
As usual we find that, as $\eps \to 0$, 
\[
p_\eps \to \rho_0(t,x) M(\Omega), \qquad M(\Omega) =\frac{1}{|{\mathbb S}^{d-1}|} \ind{{\mathbb S}^{d-1}}  \qquad \int_{{\mathbb S}^{d-1}} M
\mathrm{d}\Omega= 1,
\]
and, integrating \eqref{eq:K3} we have
\[
 \partial_t \rho_\eps (t,x) + \dv \left( \frac{F}{\mu(S)} j_\eps \right) =0,
\]
with 
\[
j_\eps (t,x) = \frac 1 \eps \int_{{\mathbb S}^{d-1}} \Omega \;  p_\eps (t,x, \Omega) \, \mathrm{d}\Omega .
\]
We compute
\[
\eps   \partial_t j_\eps +  \int_{{\mathbb S}^{d-1}}  \Omega \otimes \Omega \cdot\nabla \left [ \frac{F}{\mu(S)} p_\eps \right ]\,\mathrm{d}\Omega
= -j_\eps(t,x) .
\]
We see immediately that
$$
\int_{{\mathbb S}^{d-1}}  \Omega \otimes \Omega M(\Omega) \mathrm{d} \Omega =
\frac 1{|\mathbb{S}^{d-1}|}\int_{{\mathbb S}^{d-1}}  \Omega \otimes \Omega \,\mathrm{d} \Omega = \frac 1 d I,
$$
with $I$ the identity matrix in $\mathbb{R}^d$.

Hence, in the limit we find
\[
D \nabla  \left[ \frac{F}{\mu(S)} \rho_0 \right ] 
= -  j_0(t,x), \qquad \quad D := \frac 1d.
\]

As a consequence, we obtain the macroscopic equation
\begin{equation}
\partial_t \rho_0 - D \, \dv \left[  \frac{F}{\mu(S)} \nabla  \left( \frac{F}{\mu(S)} \rho_0 \right)\right]=0 .
\label{eq:diff}
\end{equation}
This is \blue{a special case of the general Patlak-Keller-Segel systems, with diffusion and sensitivity depending on $S$, see~\cite{HillenPainterUser, Perthame2004_chemo}.
}

\subsection{Fast tumbling, moderate friction}
\label{sec:is}

The order in which limits are taken can be important. Therefore it is useful to consider taking the limits in the opposite order. Therefore, we first set $\eps= \frac 1 \lb \approx 0$ with the cell mass $m=O(1)$. Then, we rescale Eq. \eqref{eq:K1} as
\begin{equation}\label{eq:K4}
\eps \partial_t f_\eps +V \Omega\cdot\nabla_x f_\eps + \frac 1 m \partial_V [(F- V \mu(S)) f_\eps ]  =   \frac 1 \eps  {\mathcal T}[f_\eps],
\end{equation}
and set 
\[
g_\eps (t,x,V) = \int_{{\mathbb S}^{d-1}} f_\eps (t,x,V, \Omega) \,\mathrm{d}\Omega.
\]
\red{As before,} the tumbling kernel gives
\[
f_\eps \to g_0(t,x, V) M(\Omega), \qquad  \int_{{\mathbb S}^{d-1}} \Omega f_\eps \, \mathrm{d}\Omega \to 0.
\]
By integrating Eq.~\eqref{eq:K4} in $\Omega$, we have
\begin{equation}\label{eq:geps}
 \partial_t g_\eps (t,x,V) + V \nabla_x\cdot J_\eps + \frac{1}{{\eps} m} \partial_V [(F- V \mu(S)) g_\eps ]=0, 
\end{equation}
where the flux $J_\eps$ is defined as
\begin{equation}\label{eq:Jeps}
J_\eps(t,x,V) := \frac 1 \eps \int_{{\mathbb S}^{d-1}} \Omega f_\eps \, \mathrm{d}\Omega.
\end{equation}
We compute 
\[
\eps \partial_t \int_{{\mathbb S}^{d-1}} \Omega f_\eps \, \mathrm{d}\Omega +V \int_{{\mathbb S}^{d-1}}\Omega\otimes \Omega \cdot \nabla_x f_\eps \, \mathrm{d}\Omega
+ \frac 1 m \partial_V 
\Bigg[ (F- V \mu(S))\Big( 
\int_{{\mathbb S}^{d-1}} \Omega f_\eps \, \mathrm{d} \Omega 
\Big)\Bigg] = - J_\eps,
\]
which can be written as
\[
\eps^2 \partial_t J_\eps +V \int_{{\mathbb S}^{d-1}}\Omega\otimes \Omega \cdot \nabla_x f_\eps \, \mathrm{d}\Omega
+ \frac{\eps}{m} \partial_V[
(F- V \mu(S))J_\eps ]  = - J_\eps,
\]
and thus, we find that, as $\eps\to 0$,
\[
J_\eps \to -DV\nabla_xg_0.
\]
By taking the limit $\eps\to 0$ in Eq.~\eqref{eq:geps}, we have
\[
\partial_V[(F-V\mu(S))g_0]=0,
\]
and thus,
\[
g_0=\rho_0(t,x)\delta\left(V=\frac{F}{\mu(S)}\right).
\]
By integrating Eq.~\eqref{eq:geps} in $V$, we have
\[
\partial \rho_\eps + \dv \int_{\mathbb R_+}  V J_\eps(t,x,V)dV=0.
\]
This gives
\begin{equation} \label{eq:diff2}
 \partial_t \rho_0 -D \Delta\left [ \left(\frac {F}{\mu(S)} \right)^2 \rho_0 \right]= 0.
\end{equation}

\bigskip

This is the form used for describing the so-called {\it density \red{suppressed} motility}  when the signal $S$ is related to the cell density $\rho$, thus generating a nonlinear system for which an abundant literature which treat of existence and blow-up \cite{JLaur_22,LWang_22, FujieSenba22}.

\subsection{The general scaling as $m=O(\frac 1 \lambda)$}
\label{sec:alpha}

For more generality, we may consider the same scale for $m$ and $\frac 1 \lambda$. Then, we set 
\[
\eps=\frac 1 \lambda \approx 0, \qquad  m \lambda = \alpha \qquad \alpha \in (0,\infty).
\]
and rescale Eq.~\eqref{eq:K1} as
\begin{equation}\label{eq:K5}
\eps \partial_t f_\eps +V \Omega \cdot \nabla_x f_\eps + \frac{1}{\alpha \eps} \partial_V [(F- V \mu(S)) f_\eps ]  =   \frac 1 \eps  {\mathcal T}[f_\eps],
\end{equation}

As before, integrating Eq.~\eqref{eq:K5} in $V$ and $\Omega$, we find the continuity equation for the density
\[
\partial \rho_\eps + \dv \int_{{\mathbb R}_+}  V J_\eps(t,x,V)dV=0 ,
\]
where the flux $J_\eps$ is defined by Eq.~\eqref{eq:Jeps}
and our purpose is to find the asymptotic behaviour of $J_0$.

At the leading order equation of \eqref{eq:K5}, we have
\[
   \frac 1 \alpha \partial_V\left[
        ( F-V\mu(S))  f_0
    \right]= \mathcal{T}[f_0].
\]
By integrating the above equation in $\Omega$ and $V$, we find, respectively,
\begin{equation}
 \partial_V\left[
        (  F-V\mu(S))g_0
    \right]=0,
    \qquad
    \mathcal{T}[p_0]=0.
\end{equation}
Here, we use $f_0(V=0)=0$ according to \eqref{eq:K1}.
Thus, we obtain
\begin{equation}
    g_0(t,x,V)=\rho_0(t,x)\delta \left(V=\frac{F}{\mu(S)}\right), \quad p_0(t,x,\Omega)=\rho_0(t,x)M(\Omega).
\end{equation}

This gives, for the uniform tumbling kernel \eqref{eq:T},
\begin{equation}\label{eq:f0_K5}
\frac 1 \alpha\partial_V
\left[\left(F-V\mu(S)
\right)f_0\right]=\rho_0\delta\left(V=\frac F{\mu{S}}\right)-f_0.
\end{equation}

We compute $J_\eps$ as before:
\[
\eps \partial_t \int_{{\mathbb S}^{d-1}} \Omega f_\eps \, \mathrm{d}\Omega +V \int_{{\mathbb S}^{d-1}}\Omega\otimes \Omega \cdot \nabla_x f_\eps \, \mathrm{d}\Omega
+ \frac{1}{\alpha\eps} \partial_V
\Big(
[(F- V \mu(S)) ]\int_{{\mathbb S}^{d-1}} \Omega f_\eps \, \mathrm{d} \Omega 
\Big)= - J_\eps,
\]
\[
\eps^2 \partial_t J_\eps
+V \int_{{\mathbb S}^{d-1}}\Omega\otimes \Omega \cdot \nabla_x f_\eps \, \mathrm{d}\Omega
+ \frac{1}{\alpha} \partial_V
\Big(
[(F- V \mu(S)) ]J_\eps
\Big)= - J_\eps.
\]
and thus, we find that, as $\varepsilon\to 0$,
\begin{equation}\label{eq:J0_K5}
V \nabla_x\cdot \int_{{\mathbb S}^{d-1}}\Omega\otimes \Omega\,f_0 \, \mathrm{d}\Omega
+ \frac{1}{\alpha} \partial_V
\Big(
[(F- V \mu(S)) ]J_0
\Big)= - J_0.
\end{equation}

Multiplying Eq.~\eqref{eq:J0_K5} by $V$ and \red{integrating} it in $V$, we have
\[
\nabla_x\int_{{\mathbb S}^{d-1}} \Omega\otimes\Omega \int_{{\mathbb R}_+}  V^2 f_0dVd\Omega 
-\frac 1\alpha \left(
F\int_{{\mathbb R}_+}  J_0dV-\mu(S)\int_{{\mathbb R}_+} VJ_0dV
\right)
=-\int_{{\mathbb R}_+}  VJ_0dV,
\]
which can be written in the following form:
\begin{equation}\label{eq:VJ0_K5}
\nabla_x\int_{{\mathbb S}^{d-1}}\Omega\otimes\Omega\left(\int_{{\mathbb R}_+}  V^2f_0dV\right)d\Omega+\frac F\alpha\nabla_x\int_{{\mathbb S}^{d-1}}\Omega\otimes\Omega\left(\int_{{\mathbb R}_+}  V f_0 dV\right)d\Omega
=-\left(
1+\frac{\mu(S)}{\alpha}
\right)\int_{{\mathbb R}_+}  V J_0dV.
\end{equation}
where, we use Eq.~\eqref{eq:J0_K5} recurrently,

Multiplying \eqref{eq:f0_K5} by $V$ and integrating it in $V$, we have
\[
-\frac 1\alpha
\left[Fp_0-\mu(S)\int_{{\mathbb R}_+}  Vf_0dV
\right]=\frac{F}{\mu(S)}\rho_0-\int_{{\mathbb R}_+}  Vf_0dV,
\]
\[
\left(1+\frac {\mu(S)}{\alpha} \right)\int_{{\mathbb R}_+}  Vf_0dV
=\frac{F}{\mu(S)}\rho_0+\frac{F}{\alpha}\rho_0 M(\Omega).
\]
Multiplying the above equation by $\Omega\otimes\Omega$ and integrating it in $\Omega$, we obtain
\begin{equation}\label{eq:oovf0}
\int_{{\mathbb S}^{d-1}} \Omega\otimes\Omega \left(\int_{{\mathbb R}_+}  Vf_0dV\right) d\Omega
=\frac{F}{\mu(S)}\rho_0 DI.
\end{equation}
In the same way, integrating Eq.~\eqref{eq:f0_K5} multiplied by $V^2$ in $V$, we have
\[
\left(
1+\frac{2\mu(S)}{\alpha}
\right)
\int_{{\mathbb R}_+}  V^2f_0dV
=\left(\frac{\mu(S)}{\alpha}\right)^2\rho_0
+\frac{2F}{\alpha}\int_{{\mathbb R}_+}  Vf_0dV,
\]
and integrating the above equation multiplied by $\Omega\otimes\Omega$ in $\Omega$, we obtain
\begin{equation}\label{eq:oovvf0}
\int_{{\mathbb S}^{d-1}}\Omega\otimes\Omega\left(
\int_{{\mathbb R}_+}  V^2f_0dV
\right)
d\Omega=
\left(\frac{F}{\mu(S)}\right)^2\rho_0D I.
\end{equation}
Inserting Eqs.~\eqref{eq:oovf0} and \eqref{eq:oovvf0} into Eq.~\eqref{eq:VJ0_K5}, we obtain
\begin{equation}\label{eq:vJ0_K5}
D\nabla_x\left[
\left(\frac{F}{\mu(S)}\right)^2\rho_0
\right]
+\frac{F}{\alpha}D\nabla_x
\left[
\left(
\frac{F}{\mu(S)}
\right)\rho_0
\right]
=-\left(1+\frac{\mu(S)}{\alpha}\right)\int_{{\mathbb R}_+}  VJ_0dV.
\end{equation}

This equation connects continuously \eqref{eq:diff} to \eqref{eq:diff2}. We find
\[
\int_{{\mathbb R}_+}  V J_0 dV \approx - D\frac{F}{\mu(S)} \nabla_x \left[ \left( \frac{F}{\mu(S)}\right) \rho_0\right]  \quad \text{ for } \alpha \approx 0, \qquad \text{(strong friction)},
\]
\[
 \int_{{\mathbb R}_+}  V J_0 dV \approx  -D \nabla_x \left[ \left( \frac{F}{\mu(S)}\right)^2 \rho_0\right]  \qquad \text{ for } \alpha \gg 1, \qquad \text{(fast tumbling)}.
\]
Eq.~\eqref{eq:vJ0_K5} can be rewritten as
\[
\int_{{\mathbb R}_+}  VJ_0dV=-D
\frac{F^2}{\mu(S)^2}
\left[
\nabla_x\rho_0+\rho_0
\frac{2+\frac{\mu(S)}{\alpha}}{\frac{F}{\mu(S)}+\frac{F}{\alpha}}
\nabla_x\left(\frac{F}{\mu(S)}\right)
\right].
\]

Consequently, the macroscopic equation takes the general form
\begin{equation}
\partial_t \rho_0-D\dv\left\{
\frac{F^2}{\mu(S)^2}
\left[
\nabla_x\rho_0+\rho_0
\frac{2+\frac{\mu(S)}{\alpha}}{\frac{F}{\mu(S)}+\frac{F}{\alpha}}
\nabla_x\left(\frac{F}{\mu(S)}\right)
\right]
\right\}
=0.
\end{equation}
Furthermore, when $F$ is constant, it is convenient to absorb the factor $\frac{F}{\alpha}$ into a redefined constant $\frac 1 \alpha$.
In this case, the equation reduces to
\begin{equation} \label{extendedKS}
\partial_t \rho_0 -D\dv \left\{   \frac{F^2}{\mu(S)^2} \left[
\nabla_x \rho_0\ + \rho_0  \nabla_x \log
\left(
\alpha \frac{F^2}{\mu(S)^2} + \frac{F}{\mu(S)}
\right)
\right]
\right\}=0.
\end{equation}
\blue{This class of equations makes the link between the Patlak-Keller-Segel equations and density supressed models obtained before, see \eqref{eq:diff} and \eqref{eq:diff2}. 
}

\subsection{Fast tumbling, very small friction}
\label{sec:vsf}

A very different physical situation is when friction is small. We analyze it setting
\[
\eps=\frac 1 \lambda \approx 0, \qquad  m = \frac \beta \eps \qquad \beta \in (0,\infty).
\]
Then, we rescale Eq. \eqref{eq:K1} as
\begin{equation}\label{eq:K6}
\eps \partial_t f_\eps +V \Omega\cdot \nabla_x f_\eps + \frac \eps \beta \partial_V [(F- V \mu(S)) f_\eps ]  =   \frac 1 \eps  {\mathcal T}[f_\eps].
\end{equation}
The strong tumbling kernel gives
\[
f_\eps \to g_0(t,x, V) M(\Omega), \qquad  \int_{{\mathbb S}^{d-1}} \Omega f_\eps \, \mathrm{d}\Omega \to 0.
\]
Dividing by $\eps$ and integrating Eq.~\eqref{eq:K6} in $\Omega$, we obtain
\[
 \partial_t g_\eps (t,x,V) + V \nabla_x\cdot J_\eps + \frac 1 \beta \partial_V [(F- V \mu(S)) g_\eps ]=0.
\]
We compute $J_\eps$ as before, 
\[
\eps \partial_t \int_{{\mathbb S}^{d-1}} \Omega f_\eps \, \mathrm{d}\Omega +V \int_{{\mathbb S}^{d-1}}\Omega\otimes \Omega \cdot \nabla_x f_\eps \, \mathrm{d}\Omega
+ \frac \eps \beta \partial_V
\Big(
[(F- V \mu(S)) ]\int_{{\mathbb S}^{d-1}} \Omega f_\eps \, \mathrm{d} \Omega 
\Big)= - J_\eps,
\]
and conclude
\begin{equation} \label{eq:g}
 \partial_t g_0 - D V^2 \Delta g_0 + \frac 1 \beta \partial_V [(F- V \mu(S)) g_0 ]=0.
\end{equation}
We notice that the limit  $\beta \to 0$ gives Eq.~\eqref{eq:diff2}.

\section{Bounds for the macroscopic equation} 

For the mathematical analysis, it is convenient to change the notation and introduce the equilibrium velocity

\begin{equation} \label{def:v}
v(S):= \frac{F}{\mu(S)}> 0.
\end{equation}
Then, we work with Eq.~\eqref{extendedKS} written as 
\begin{equation} \label{Basic}
\partial_t \rho -D\dv \Big[   v(S)^2 \Big(   \nabla \rho + \rho  \nabla \log [ \alpha v(S)^2+ v(S) ] \Big)  \Big]=0 .
\end{equation}

It may be interesting to study some  estimates in time of suitable functionals involving the solution $\rho$ of \eqref{Basic}.

As a first step, we study the entropy property,
detailed in the following proposition.

\begin{proposition}[Entropy control] \label{prop:entropy}
Being given $T>0$, with the condition on $v:=v(S)$
\begin{equation} \label{entrpieCond}
\sup_{0\leq t \leq T} \Big \| \frac{|\nabla_x v|^2}{v^2} \Big \|_{_{{L ^{{d}/{2}}(\mathbb{R}^d})}} < \infty,
\end{equation}
the strong solutions of \eqref{Basic}, 
posed in $\mathbb{R}^d$, with initial data $\rho^0 >0$, satisfy the following entropy estimate for all $t\leq T$:
\[
\int_{\mathbb{R}^d} \rho(t) \log \rho(t)\,\mathrm{d}x+D \int_0^t \! \int_{\mathbb{R}^d}  |\nabla(v\sqrt{\rho})|^2\,\mathrm{d}x\,\mathrm{d}t  \leq \int_{\mathbb{R}^d} \rho^0 \log \rho^0 \,\mathrm{d}x+C t.
\]
\end{proposition}

\begin{proof}
We compute, using mass conservation, 
\begin{align*}
\frac{\mathrm{d}}{\mathrm{d}t} \int_{\mathbb{R}^d} \rho \log(\rho)\,\mathrm{d}x
= D \int_{\mathbb{R}^d}\log (\rho)\dv(v^2\nabla \rho)\,\mathrm{d}x+D\int_{\mathbb{R}^d} \log(\rho)\dv [v^2\rho\nabla \log(\al v^2+v)]\,\mathrm{d}x.
\end{align*}
\blue{ The structure here differs from the usual form of a dispersive term and a hyperbolic concentration term that occur in the Patlak-Keller-Segel system and which can be handled by Gagliardo-Nirenberg-Sobolev inequalities, see \cite{FujieSenba22, Perthame2004_chemo, BDP2006KS}.}  This leads us to treat separately these two terms. Integrating by parts, the first term gives
\begin{equation*}
\begin{split}
-\int_{\mathbb{R}^d} \log (\rho)\dv(v^2\nabla \rho)\,\mathrm{d}x &=
\int_{\mathbb{R}^d}\left|\frac{v\nabla \rho}{\sqrt{\rho}}\right|^2
\,\mathrm{d}x
=4\int_{\mathbb{R}^d} |v\nabla \sqrt{\rho}|^2
\,\mathrm{d}x
=4\int_{\mathbb{R}^d} \left |\nabla(v\sqrt{\rho})-\sqrt{\rho}\nabla v\right|^2
\,\mathrm{d}x\\
&=4\int_{\mathbb{R}^d}\Big [
|\nabla(v\sqrt{\rho})|^2
+|\sqrt{\rho}\nabla v|^2
-2\sqrt{\rho} \nabla v\nabla(v\sqrt{\rho})
\Big]\,\mathrm{d}x.
\end{split}
\end{equation*}
The second term becomes
\begin{equation*}
\begin{split}
-\int_{\mathbb{R}^d} \log\rho\dv [v^2\rho\nabla \log(\al v^2+v)]\,\mathrm{d}x
&=2\int_{\mathbb{R}^d} v^2\sqrt{\rho}\nabla\sqrt{\rho}
\left(
\frac{\nabla v}{v}+\frac{\al \nabla v}{\al v +1}
\right)\,\mathrm{d}x\\
&=2\int_{\mathbb{R}^d} v\sqrt{\rho}\nabla v \nabla \sqrt{\rho}\left(
1+\frac{\al v}{\al v+1}
\right)\,\mathrm{d}x\\
&=2\int_{\mathbb{R}^d} \left(2-\frac{1}{\al v+1}\right)
\sqrt{\rho}\nabla v(v\nabla\sqrt{\rho})\,\mathrm{d}x\\
&=2\int_{\mathbb{R}^d} \left(2-\frac{1}{\al v+1}\right)
\sqrt{\rho}\nabla v[\nabla (v\sqrt{\rho})
-\sqrt{\rho}\nabla v]\,\mathrm{d}x\\
&=4\int_{\mathbb{R}^d} \left(1-\frac{1}{2(\al v+1)}\right)
[\sqrt{\rho}\nabla v\nabla (v\sqrt{\rho})
-|\sqrt{\rho}\nabla v|^2]\,\mathrm{d}x.
\end{split}
\end{equation*}
As a consequence, we conclude that
\begin{equation} \label{entrpieDecay}
    \frac{d}{dt}\int_{\mathbb{R}^d} \rho \log\rho \,\mathrm{d}x
    =-4D\int_{\mathbb{R}^d}\left[  |\nabla(v\sqrt{\rho})|^2
    +\frac{|\sqrt{\rho}\nabla v|^2}{2(\al v +1)}
    -\sqrt{\rho}\nabla v\nabla(v\sqrt{\rho})
    \left(
    1+\frac{1}{2(\al v+1)}
    \right)
    \right]\,\mathrm{d}x.
\end{equation}

In order to deduce a control of entropy from \eqref{entrpieDecay}, we control the cross term as 
\[
\int_{\mathbb{R}^d} \sqrt{\rho}\nabla v\nabla(v\sqrt{\rho}) \left( 1+\frac{1}{2(\al v+1)} \right)\,\mathrm{d}x
\leq \frac 12 \int_{\mathbb{R}^d} |\nabla(v\sqrt{\rho})|^2\,\mathrm{d}x+ \frac 12 \int_{\mathbb{R}^d} \rho |\nabla v|^2 \left( 1+\frac{1}{2(\al v+1)} \right)^2\,\mathrm{d}x.
\]
This gives, on the one hand
\[
\begin{split}
  \frac{1}{2D}  \frac{d}{dt}\int_{\mathbb{R}^d} \rho \log\rho\,\mathrm{d}x
    &+\int_{\mathbb{R}^d}  |\nabla(v\sqrt{\rho})|^2\,\mathrm{d}x
    \leq \int_{\mathbb{R}^d} \left[- \frac{|\sqrt{\rho}\nabla v|^2}{\al v +1}+ \rho |\nabla v|^2 \left( 1+\frac{1}{2(\al v+1)} \right)^2 \right]\,\mathrm{d}x
 \\
 &= \int_{\mathbb{R}^d} \rho v^2 \frac{|\nabla v|^2}{v^2} \left( 1+\frac{1}{4(\al v+1)^2} \right)\,\mathrm{d}x.
\end{split} 
\]
On the other hand, we may set $w=v\sqrt{\rho}$ and use the Sobolev inequality in dimension $d$:
\[
\| w \|_{L^{\frac{2d}{d-2}}(\mathbb{R}^2)} \leq C(d) \|\nabla w \|_{L^2(\mathbb{R}^d)}. 
\]
In the  above inequality, this gives, because \(\| \rho v^2\|_{
{L^{\frac {d}{d-2}}(\mathbb{R}^d)}} = \| \sqrt \rho v\|_{
{L^{\frac{2d}{d-2}}(\mathbb{R}^d)}}^2\), 
\begin{align*}
 \frac{1}{2D}  \frac{d}{dt}\int_{\mathbb{R}^d} \rho \log \rho \,\mathrm{d}x+\int_{\mathbb{R}^d}   |\nabla(v\sqrt{\rho})|^2\,\mathrm{d}x& \leq
 \| \rho v^2\|_{L^{\frac{d}{d-2}}(\mathbb{R}^2)} 
 \left \| \frac{|\nabla v|^2}{v^2} \left( 1+\frac{1}{4(\al v+1)^2} \right)\right\|_{L^{d/2}(\mathbb{R}^2)}
\\
& \leq C(d)  \|\nabla v\sqrt{\rho} \|_{L^{2}(\mathbb{R}^d)} 
\left\| \frac{|\nabla v|^2}{v^2} \left( 1+\frac{1}{4(\al v+1)^2} \right)
\right\|_{L^{d/2}(\mathbb{R}^d)}
\\
& \leq \frac 12 \|\nabla v\sqrt{\rho} \|_{L^{2}(\mathbb{R}^d)}^2+ C(d) \left( \left\| \frac{|\nabla v|^2}{v^2} \left( 1+\frac{1}{4(\al v+1)^2} \right)\right\|_{L^{d/2}(\mathbb{R}^d)}\right)^2.
\end{align*} 
This ends the proof of \red{Proposition}~\ref{prop:entropy}.
\end{proof}

The second estimate on the long-time behavior of the solution is proved in a smooth, connected, bounded domain $\mathcal{O}\subset \mathbb{R}^d$, with a zero-flux boundary condition,
i.e., we suppose that
\begin{equation}\label{eq:bc}
n\cdot v(S)^2\big(\nabla \rho + \rho \,\nabla \log(\alpha v(S)^2+v(S))\big)=0
\quad\text{on }\partial\mathcal{O},
\end{equation}
with $n$ is the outward unit normal.

From now on, we assume that $S$ is given and that $v(S)$ is smooth, time-independent, and strictly positive. Sometimes we denote it as $v=v(x)$ for abuse of notation.

It is clear that the initial-boundary value problem \eqref{Basic}-\eqref{eq:bc} conserves the mass of the solution.
Indeed, integrating \eqref{Basic} over $\mathcal{O}$ and using the divergence theorem together
with the boundary condition \eqref{eq:bc} yields
$$
\frac{\mathrm{d}}{\mathrm{d}t}\int_\mathcal{O} \rho(x,t)\,\mathrm{d}x
= D\int_\mathcal{O}\nabla\cdot\big[v^2\big(\nabla\rho+\rho\nabla\log(\alpha v(S)^2+v(S))\big)\big]\,\mathrm{d}x
$$
$$
= D\int_{\partial\mathcal{O}} n\cdot v^2(\nabla\rho+\rho\nabla\log(\alpha v(S)^2+v(S))\,
\mathrm{d}\sigma = 0,
$$
where $\mathrm{d}\sigma$ denotes the surface element on
$\partial\mathcal{O}$.
Hence, the total mass is conserved in time.
\begin{equation}\label{eq:def-M}
M:=\int_\mathcal{O}\rho(x,t)\,\mathrm{d}x = \int_\mathcal{O}\rho^0(x)\,\mathrm{d}x 
\end{equation}

Let now denote
\begin{equation}
\label{eq:def-psi}
\psi(x):=\log\big(\alpha v(S(x))^2+v(S(x))\big), \quad \text{and} \quad Z := \int_{\mathcal{O}} e^{-\psi(x)}\,\mathrm{d}x, 
\end{equation}
and investigate the long-time behavior of the following functional, which is  indicated, by the analogy with a physical system, the free energy:
\begin{equation}
\label{eq:def-energy}
\mathcal{E}[\rho]:=
\int_\mathcal{O} \rho(x,t)\,\log\!\big(\rho(x,t)\,[\alpha v(x)^2+v(x)]\big)\,\mathrm{d}x
= \int_\mathcal{O} \rho(x)\,\big(\log\rho(x)+\psi(x)\big)\,\mathrm{d}x.
\end{equation}
The following results hold.

\begin{proposition} [Free energy decay]
Assume $S$ is independent of time. Let $\rho >0$ be the strong solutions of \eqref{Basic}  defined in a smooth domain $\mathcal{O}\subset \mathbb{R}^d$ with boundary
condition \eqref{eq:bc} and initial data $\rho^0 >0$.
Then, it holds
$$
\frac{d}{dt}\mathcal{E}[\rho]
\leq 0.
$$
\end{proposition}

\begin{proof}
Assuming $\rho>0$ and sufficient smoothness, we differentiate in time and use \eqref{Basic}:
\begin{align}
\frac{d}{dt}\mathcal{E}[\rho]
&= \int_\mathcal{O}\big(\log\rho+1+\psi\big)\,\partial_t\rho\,\mathrm{d}x
= D\int_\mathcal{O}\big(\log\rho+1+\psi\big) \!\nabla\!\cdot\!\big(v^2(\nabla\rho+\rho\nabla\psi)\big)\,\mathrm{d}x \notag\\
&= -D\int_\mathcal{O} v^2\,\nabla(\log\rho+\psi)\cdot(\nabla\rho+\rho\nabla\psi)\,\mathrm{d}x,
\label{eq:pre-diss}
\end{align}
where the boundary term vanishes by~\eqref{eq:bc}. Using the identity $\nabla\rho=\rho\,\nabla\log\rho$,
we recognize the perfect square
$$
\nabla(\log\rho+\psi)\cdot(\nabla\rho+\rho\nabla\psi)
= \rho\,\big|\nabla(\log\rho+\psi)\big|^2,
$$
and obtain that
\begin{equation}
\label{eq:free-energy-evolution}
\frac{d}{dt}\mathcal{E}[\rho]
= -D\int_\mathcal{O} v(x)^2\,\rho(x,t)\,\big|\nabla(\log\rho+\psi)\big|^2\,\mathrm{d}x
\end{equation}
The right-hand side of the previous equality being non-\red{positive}, we deduce that $\mathcal{E}[\rho]$ is nonincreasing along smooth solutions
of \eqref{Basic}-\eqref{eq:bc}.
\end{proof}

The next step consists in proving that the free energy functional $\mathcal{E}[\rho]$ is uniformly lower bounded.

The following proposition holds.

\begin{proposition}
\label{prop:v_independent_2}
Assume $0<Z<\infty$.
Then, for every measurable $\rho\geq 0$, solution of \eqref{Basic}-\eqref{eq:bc} with mass $M>0$, for which $\mathcal{E}[\rho]$ is well-defined in $(-\infty,+\infty]$, the following lower bound holds:
\begin{equation}
\label{eq:main-lower}
  \mathcal{E}[\rho] \geq M\log \Big(\frac{M}{Z}\Big).
\end{equation}
Moreover, equality in \eqref{eq:main-lower} holds if and only if
\begin{equation} 
\label{eq:minimizer}
\rho^*(x) = \frac{M}{Z}\,e^{-\psi(x)} \qquad \text{for a.e.\ }x\in\mathcal{O}.
\end{equation}
\end{proposition}

\begin{proof}
Introduce the probability densities
$$
\nu := \frac{\rho}{M}, \qquad \text{and} \qquad \pi := \frac{e^{-\psi}}{Z}.
$$
By construction, both $\nu$ and $\pi$ are nonnegative.
Define the relative entropy
$$
D(\nu\vert\pi) := \int_{\mathcal{O}}\nu(x)\,\log \left(\frac{\nu(x)}{\pi(x)}\right) \mathrm{d}x 
$$
with the convention $D(\nu\vert\pi)=+\infty$ if $\nu$ is not absolutely continuous with respect to $\pi$.

A direct computation yields
\begin{align*}
\mathcal{E}[\rho]
&= \int_{\mathcal{O}} \rho(x)\,\log \left(\frac{\rho(x)}{e^{-\psi(x)}}\right)\,\mathrm{d}x \\
&= M \int_{\mathcal{O}} \nu(x)\,\log \left(\frac{\nu(x)}{\pi(x)}\right)\,\mathrm{d}x+ M\log\left (\frac{M}{Z}\right) 
\\
&= M\,D(\nu\vert\pi)+M\log \left(\frac{M}{Z}\right).
\end{align*}
Now, by Jensen's inequality, the relative entropy satisfies, with $f(x)= \frac{\nu(x)}{\pi(x)}$, 
$$
D(\nu\vert\pi) \geq \int_{\mathcal{O}} f(x) \pi(x)\,\mathrm{d}x \ln \left( \int_{\mathcal{O}} f(x) \pi(x)\,\mathrm{d}x \right) \geq 0.
$$ 
Hence we deduce the lower bound \eqref{eq:main-lower}. 

Equality holds if and only if $D(\nu\vert\pi)=0$, i.e., $\nu=\pi$ almost everywhere, which is equivalent to \eqref{eq:minimizer}.
\end{proof}

The next lemma identifies the minimizer of the previous proposition with the stationary state of the system. This identification paves the way to study the long-time convergence to equilibrium for the initial-boundary value problem \eqref{Basic}--\eqref{eq:bc}, with a given \red{strictly} positive initial condition.

\begin{lemma}
\label{lem:stationary}
Assume that $v(x)>0$ in $\mathcal{O}$.
The unique strictly positive stationary solution of \eqref{Basic}--\eqref{eq:bc} with prescribed mass $M>0$ is
$$
\rho_\infty(x)=\frac{M}{Z}\,e^{-\psi(x)}=\rho^*(x)
\qquad\text{for a.e.\ }x\in\mathcal{O},
$$
where $\rho^*$ is the minimizer identified in \eqref{eq:minimizer}.
\end{lemma}

\begin{proof}
We proceed by analysis-synthesis. Suppose that there exists a strictly positive stationary solution $\rho_\infty$.
Hence, we have $\partial_t\rho_\infty=0$ and, from Eq. \eqref{Basic}, we recognize that $\rho_\infty$ satisfies
$$
\nabla\cdot \big(v^2(\nabla\rho_\infty+\rho_\infty\nabla\psi)\big) =0 \quad \text{in }\mathcal{O}.
$$
Using \eqref{eq:pre-diss} and \eqref{eq:free-energy-evolution}, we deduce that
$$
0=\frac{d}{dt}\mathcal{E}[\rho_\infty]
= -D\int_\mathcal{O} v(x)^2\,\rho_\infty(x)\,\big|\nabla(\log\rho_\infty+\psi)\big|^2\,\mathrm{d}x.
$$
Since $D>0$, $v^2>0$ and $\rho_\infty>0$, the integrand is nonnegative and the integral vanishes only if
$$
\nabla(\log\rho_\infty+\psi)=0 \quad \text{a.e.\ in }\mathcal{O}.
$$
Imposing the mass constraint 
$$
\int_{\mathcal{O}}\rho_\infty\mathrm{d}x=M,
$$ 
we deduce 
$$
\rho_\infty(x)=\frac{M}{Z}e^{-\psi(x)}=\rho^*(x).
$$

Conversely, for $\rho=\rho^*$ one has $\nabla(\log\rho^*+\psi)=0$, which implies
$$
v^2(\nabla\rho^*+\rho^*\nabla\psi)=0, \qquad \text{in} \quad \mathcal{O}
$$
 In particular,
$\nabla\cdot\big(v^2(\nabla\rho^*+\rho^*\nabla\psi)\big)=0$ in $\mathcal{O}$, and the boundary condition \eqref{eq:bc} is satisfied. The uniqueness among positive densities with mass $M$ follows from the above characterization.
\end{proof}

We are now ready to analyze the long-time behavior of the initial-boundary value problem \eqref{Basic}--\eqref{eq:bc}. Our approach is based on quantitative estimates of the relative entropy with respect to the stationary state (the entropy-dissipation method, see for \red{example} \cite{MR2409469, MR1853037, MR2823993}), which allow to deduce quantitative decay bounds that yield convergence of solutions to equilibrium over long times.

The first technical result is the following.
\begin{lemma}
\label{lem:Poincare}
Let $\mathcal{O}\subset\mathbb{R}^d$ be a smooth, bounded, connected domain with outward unit normal $n$.
Assume $v$ is smooth, time-independent and satisfies
$$
0<v_{\textrm{min}} \leq v(x)\leq v_{\textrm{max}}<\infty \qquad \text{for all }x\in\overline{\mathcal{O}}.
$$
Let $\rho_\infty$ be the stationary state identified in Lemma~\ref{lem:stationary}.
Then, for every $u\in H^1(\mathcal{O})$ such that
\begin{equation}
\label{eq:zero-average}
\int_{\mathcal{O}}u(x)\,\rho_\infty(x)\, \mathrm{d}x=0,
\end{equation}
the following weighted Poincar\'e inequality holds:
\begin{equation}
\label{eq:WP}
\int_{\mathcal{O}} u^2\,\rho_\infty\, \mathrm{d}x \leq  \frac{(\alpha v_{\max}^2+v_{\max})^2 C_{\mathcal{O}}}{v_{\textrm{min}}^2(\alpha v_{\min}^2+v_{\min})^2} \int_{\mathcal{O}} v(x)^2\,\rho_\infty(x)\,|\nabla u|^2\, \mathrm{d}x,
\end{equation}
where $C_{\mathcal{O}}$ is the Poincar\'e constant of the domain $\mathcal{O}$.
\end{lemma}
\begin{proof}
Since $v_{\textrm{min}}\leq v\leq v_{\textrm{max}}$, combining the bounds for $e^{-\psi}$ and $M/Z$, we obtain the following explicit two-sided bounds for $\rho_\infty$:
\begin{equation}
\label{eq:bounds_rho}
0<\frac{M}{|\mathcal{O}|}\,
\frac{\alpha v_{\min}^2+v_{\min}}{\alpha v_{\max}^2+v_{\max}}
\leq 
\rho_\infty(x)
\leq 
\frac{M}{|\mathcal{O}|}\,
\frac{\alpha v_{\max}^2+v_{\max}}{\alpha v_{\min}^2+v_{\min}},
\end{equation}
for all $x\in\overline{\mathcal{O}}$.

In what follows, we denote, for simplicity,
$$
m_* := \frac{M}{|\mathcal{O}|}\,
\frac{\alpha v_{\min}^2+v_{\min}}{\alpha v_{\max}^2+v_{\max}},
\qquad
m^* := \frac{M}{|\mathcal{O}|}\,
\frac{\alpha v_{\max}^2+v_{\max}}{\alpha v_{\min}^2+v_{\min}}.
$$

Let $u\in H^1(\mathcal{O})$ satisfy \eqref{eq:zero-average}, and denote the mean
$$
\overline{u}:=\dfrac{1}{|\mathcal{O}|}\int_{\mathcal{O}}u\, \mathrm{d}x.
$$
We start from the identity
$$
\int_{\mathcal{O}} u^2\,\rho_\infty\, \mathrm{d}x
= \int_{\mathcal{O}} (u-\overline{u})^2\,\rho_\infty\, \mathrm{d}x
+ 2\,\overline{u}\int_{\mathcal{O}}(u-\overline{u})\,\rho_\infty\, \mathrm{d}x
+ \overline{u}^2\int_{\mathcal{O}} \rho_\infty\, \mathrm{d}x.
$$
Using \eqref{eq:zero-average}, we deduce
$$
\int_{\mathcal{O}} u^2\,\rho_\infty\, \mathrm{d}x
= \int_{\mathcal{O}} (u-\overline{u})^2\,\rho_\infty\, \mathrm{d}x
- \overline{u}^2\,M
\leq  \int_{\mathcal{O}} (u-\overline{u})^2\,\rho_\infty\, \mathrm{d}x.
$$
By the bounds \eqref{eq:bounds_rho} we have,
$$
\int_{\mathcal{O}} (u-\overline{u})^2\,\rho_\infty\, \mathrm{d}x
\leq m^* \int_{\mathcal{O}} (u-\overline{u})^2\, \mathrm{d}x
\leq m^* C_{\mathcal{O}} \int_{\mathcal{O}} |\nabla u|^2\, \mathrm{d}x 
\leq 
\frac{m^* C_{\mathcal{O}}}{v_{\textrm{min}}^2\,m_*}\int_{\mathcal{O}} v(x)^2\,\rho_\infty(x)\,|\nabla u|^2\, \mathrm{d}x,
$$
where the last steps are the consequence of the standard Poincar\'e inequality with Poincar\'e constant $C_{\mathcal{O}}$ and of the boundedness of both
$v$ and $\rho_\infty$.
Collecting all the estimates, we finally deduce \eqref{eq:WP}.
\end{proof}
We are now ready to study the long-time behavior of \eqref{Basic}--\eqref{eq:bc}.

\begin{proposition}
\label{prop:exp-conv}
Assume that $v$ satisfies the hypotheses of Lemma \ref{lem:Poincare} and let $\rho_\infty$ be the stationary state introduced in Lemma \ref{lem:stationary}, with mass $M=\Vert \rho^0\Vert_{L^1(\mathcal{O})}$.
Define
$$
u(x,t):=\frac{\rho(x,t)}{\rho_\infty(x)}-1,
$$
where $\rho$ is a sufficiently regular, strictly positive solution of \eqref{Basic}--\eqref{eq:bc} with initial condition $\rho^0$. Then $u$ satisfies for all $t\geq 0$
\begin{equation}
\label{eq:L2-decay-direct}
\|u(t)\|_{L^2(\mathcal{O};\rho_\infty(x)\mathrm{d}x)}^2 \leq  \left\Vert
\frac{\rho^0(x)}{\rho_\infty(x)}-1  \right\Vert_{L^2(\mathcal{O};\rho_\infty(x)\mathrm{d}x)}^2 \exp\left(
- 2D \frac{v_{\textrm{min}}^2(\alpha v_{\min}^2+v_{\min})^2} {(\alpha v_{\max}^2+v_{\max})^2 C_{\mathcal{O}}}
\,t\right) .
\end{equation}

\end{proposition}

\begin{proof}
It is clear that, by mass conservation, $u$ satisfies for all $t\geq 0$,
$$
\int_{\mathcal{O}}u(x,t)\,\rho_\infty(x)\, \mathrm{d}x=0 .
$$ 

From Lemma~\ref{lem:stationary} we have $\nabla(\log\rho_\infty+\psi)=0$, i.e.
$$
\nabla\rho_\infty+\rho_\infty\nabla\psi=0.
$$
Setting $\rho=\rho_\infty(1+u)$ in \eqref{Basic} and using the previous identity, we deduce that $u$ solves the problem
\begin{equation}\label{eq:u}
\rho_\infty \partial_t u = D \nabla \cdot \big(v^2\,\rho_\infty\,\nabla u\big)
\quad\text{in }\mathcal{O},
\qquad
n\cdot\big(v^2 \rho_\infty \nabla u\big)=0 \quad\text{on }\partial\mathcal{O},
\end{equation}
with initial condition 
\begin{equation}
\label{eq:u-ic}
u(x,0)=u^0 := \frac{\rho^0(x)}{\rho_\infty(x)}-1.
\end{equation}
Consider the weighted energy of the initial-boundary value problem \eqref{eq:u}-\eqref{eq:u-ic}
$$
\Phi(t):=\frac12\int_{\mathcal{O}} u(x,t)^2 \rho_\infty(x)\, \mathrm{d}x
=\frac12\|u(t)\|_{L^2(\mathcal{O};\rho_\infty(x)\mathrm{d}x)}^2.
$$
Differentiating in time and using \eqref{eq:u}, integration by parts with the Neumann boundary condition gives
$$
\Phi'(t)
= \int_{\mathcal{O}} u\,\rho_\infty\,\partial_t u\, \mathrm{d}x
= D\int_{\mathcal{O}} u\,\nabla\!\cdot\!\big(v^2\,\rho_\infty\,\nabla u\big)\, \mathrm{d}x
= -D\int_{\mathcal{O}} v(x)^2\,\rho_\infty(x)\,|\nabla u|^2\, \mathrm{d}x.
$$
Applying the weighted Poincar\'e inequality \eqref{eq:WP} to $u(t)$, we deduce
$$
\Phi'(t)\leq  -D \frac{v_{\textrm{min}}^2\,m_*}{m^* C_{\mathcal{O}}}
\int_{\mathcal{O}} u^2(x,t)\rho_\infty(x) \mathrm{d}x = -2D \frac{v_{\textrm{min}}^2\,m_*}{m^* C_{\mathcal{O}}} \Phi(t).
$$
Finally, Gr\"onwall's lemma yields \eqref{eq:L2-decay-direct}.
\end{proof}

From the previous lemma we can deduce the following result, which gives an explicit, though non necessarily optimal, estimate of the convergence rate to the stationary state of the system.

\begin{corollary}[Exponential convergence]
\label{cor:L1decay}
Under the assumptions of Proposition \ref{prop:exp-conv}, let $\rho$ be the solution of \eqref{Basic}--\eqref{eq:bc} with initial condition $\rho^0$ and $\rho_\infty$ the stationary solution with the same mass as~$\rho^0$. Then, the following decay estimate of the $L^1$-norm of the difference between $\rho$ and the stationary solution $\rho_\infty$ holds:
\begin{equation}
\label{eq:L1-decay}
\|\rho(t)-\rho_\infty\|_{L^1(\mathcal{O})} \leq  M^{1/2}
 \left\Vert
\frac{\rho^0(x)}{\rho_\infty(x)}-1  \right\Vert_{L^2(\mathcal{O};\rho_\infty(x)\mathrm{d}x)}
\exp\left(
- D \frac{v_{\textrm{min}}^2(\alpha v_{\min}^2+v_{\min})^2} {(\alpha v_{\max}^2+v_{\max})^2 C_{\mathcal{O}}}
\,t\right)
\end{equation}
for all $t\geq 0$.
\end{corollary}

\begin{proof}
We apply the Cauchy--Schwarz inequality to the $L^2$-norm of $u$ with respect to the measure $\rho_\infty  \mathrm{d}x$:
$$
\|\rho(t)-\rho_\infty\|_{L^1(\mathcal{O})} = \|u(t)\|_{L^1(\mathcal{O};\rho_\infty(x)\mathrm{d}x)}
\leq \left(\int_{\mathcal{O}}u(x,t)^2 \rho_\infty(x) \mathrm{d}x\right)^{1/2}
\left(\int_{\mathcal{O}}\rho_\infty(x) \mathrm{d}x\right)^{1/2}
$$
$$
= M^{1/2}\,\|u(t)\|_{L^2(\mathcal{O};\rho_\infty(x)\mathrm{d}x)}.
$$
Combining the last inequality with \eqref{eq:L2-decay-direct}, we obtain \eqref{eq:L1-decay}.
\end{proof}

\begin{remark}
Unlike Proposition \ref{prop:entropy}, which allows for time-dependent $v$, the study of the convergence to equilibrium in a bounded domain $\mathcal{O}$ requires $v$ (or, equivalently, $S$) to be independent of $t$. This assumption preserves the generality of the result: allowing $v$ (or $S$) to vary in time would impose us to specify the evolution law of $v$, which is not needed in Proposition \ref{prop:entropy} (where only condition \eqref{entrpieCond} is required).
\end{remark}

\section{Pattern formation}
\label{sec:patterns}

When the mechanical properties of the substrate are modified by the cells themselves, we can couple the macroscopic equation \eqref{extendedKS} with a production rule for the signal
\begin{equation} \label{convolS}
S=K*\rho ,
\end{equation}
with a convolution kernel $K(x)\geq 0$ of mass $1$ so that $$\int_{\mathbb{R}^d} S 
\,\mathrm{d} x= \int_{\mathbb{R}^d} \rho 
\,\mathrm{d}x.$$

To study the stability of the constant steady state and the pattern formation ability of Eq.~(\ref{Basic}), we work on the full space throughout this section.

\subsection{Linear instability}

We linearize Eq.~(\ref{Basic}) around the constant state $(\rho,S)=(1,1)$.
We write $\rho=1+r(t,x)$ and $S=1+s(t,x)$ and put them into Eq.~(\ref{Basic}), then we have
\begin{equation*}
\begin{split}
\partial_t r-D\dv\Big[&
\big( v(1)^2+2v(1)v'(1)s+o(s)\big) \Big\{\nabla r
\\
& 
+(1+r) \nabla\ln \big[ \al v(1)^2+v(1)+(2\al v(1)v'(1) +v'(1))s+o(s) \big]\Big\}
\Big]=0,
\end{split}    
\end{equation*}  
which can also be reformulated as
\begin{equation*}    
    \partial_t r-D\dv\left[
    \big(v(1)^2+2v(1)v'(1)s+o(s) \big)\Big \{
    \nabla r+(1+r)\nabla
    \ln\left(1+\frac{2\al v(1)v'(1) +v'(1)}{\al v(1)^2+v(1)}s+o(s)
    \right)\Big\}
    \right]=0.
\end{equation*}
Thus, the linearized equation of Eq.~(\ref{Basic}) is written as
\begin{equation}\label{eq_linearized}
    \partial_t r-Dv(1)^2 \Delta \left(
    r+\frac{v'(1)(2\al v(1)+1)}{v(1)(\al v(1)+1)} s
    \right)=0.
\end{equation}

By taking the Fourier transform of the above equation and Eq.~\eqref{convolS} for $S$,
we obtain
\begin{equation} \label{eqlin}
    \partial_t \widetilde r(k)=-Dv(1)^2 |k|^2 \big[1+\gamma_\alpha  \widetilde K(k)\big] \widetilde r(k),
\end{equation}
where the parameter $\gamma_\alpha$ is defined as 
\begin{equation}\label{eq_gamma}
\gamma_\alpha=\frac{v'(1)(2\al v(1)+1)}{v(1)(\al v(1)+1)}.
\end{equation}
As a conclusion, when $1+\gamma_\alpha  \widetilde K(k) >0$ for all $k$ we have linear stability, and if there are values of $k$ with $1+\gamma_\alpha  \widetilde K(k)<0$ we expect pattern formation.
We also note that $|\gamma_\alpha|$ increases with the parameter $\alpha$ and $|\gamma_\infty|=2|\gamma_0|$. 
Thus, when $v'(1)<0$, pattern formation more likely occurs as $\alpha$ increases.
\\

For example, when we consider a diffusion equation for $S$, 
\begin{equation}\label{eqS}
-D_S\Delta S + S=\rho,
\end{equation}
then $$\widetilde K(k)=\frac{1}{1+D_S |k|^2}$$ and Eq. \eqref{eqlin} becomes 
\[
    \partial_t \widetilde r(k)=-Dv(1)^2 |k|^2 \frac{1+\gamma_\alpha +D_S |k|^2}{1+D_S |k|^2}\widetilde r(k).
\]
This indicates that when $\gamma_\alpha <-1$, the oscillations with low wave numbers,
\begin{equation}\label{eqInsta}
0<|k|<\sqrt{-\frac{1}{D_S}\left(\gamma_\alpha+ 1\right)},
\end{equation}
become linearly unstable. Thus, we can expect pattern formation in accordance with a type of Turing instability mechanism.

\subsection{Steady periodic solutions in 1D}
\label{sec:periodic}

In one dimension, we can prove the existence of specific patterns, namely  periodic solutions. So we consider Eq.~\eqref{Basic}  coupled with Eq.~(\ref{eqS}) set in the full line.
To construct the spatially periodic steady state solutions, we argue as follows.

We denote the spatial derivative by $S_x$. We choose the reference value $x=0$ such that $S_{x}(0)=0$. Indeed, we may always choose a local minimum of $S$ at $x=0$, leading to the conditions
\begin{equation}\label{Cond0}
S_{x}(0)=0, \qquad D_S S_{xx}(0)= S_0 - \rho_0>0, 
\end{equation}
being given $\rho_0:=\rho(0)$, $S_0:=S(0)$. 
We also assume $S$ is symmetric at $x=0$, i.e., $S_{xxx}(0)=0$, which gives $\rho_x(0)=0$.
Next, we compute the steady state solutions of 
Eq.~\eqref{Basic} written as 
\begin{equation} \label{Basic2}
\partial_t \rho -D\partial_x \Big\{  \rho v(S)^2 \partial_x \ln \big[\rho \big( \alpha v(S)^2+ v(S) \big) \big]   \Big\}=0 .
\end{equation}
Since $v(S)>0$, according to \eqref{def:v}, they are given by
\[
\partial_x \ln[\rho(\al v(S)^2+v(S))]=0.
\]
Thus, we can obtain an explicit relation between  $\rho(x)$ and $S(x)$
\begin{equation}\label{eq_C0}
\rho=\frac{C_0}{\al v(S)^2+v(S)}, \qquad \qquad C_0=\rho_0 \big(\alpha v(S_0)^2+v(S_0)\big).
\end{equation}

\blue{\begin{proposition} [Existence of periodic solutions] Let $S$ be given by \eqref{eqS} and define $C_0$ by \eqref{eq_C0}. Assume $S\mapsto v(S)$ is continuously differentiable and satisfies the decay condition at infinity
\begin{equation}\label{as_per_v}
Sv(S)\le C_1<C_0\qquad \text{at} \quad  S\approx+\infty \qquad  \text{(and thus $v(S)\approx 0$)},
\end{equation}
then Eq.~\eqref{Basic}  coupled with Eq.~(\ref{eqS}) has at least  one (non-constant) periodic solution. It is unique (up to translation) when $Sv(S)$ decays for all $S$ \red{large enough} (see \eqref{as:per_v2}).
\end{proposition}
}

\begin{proof}
We can reduce the equation for $S$ as follows. Multiplying Eq.~\eqref{eqS} by $S_x$, we obtain 
\begin{equation*}
\begin{split}
    &-D_S S_x S_{xx}+S_xS=-\frac{D_S}{2}((S_x)^2)_x
    +\frac 12(S^2)_x= \rho S_x =\frac{C_0 S_x}{\alpha v(S)^2+v(S)}.
    \end{split}
\end{equation*}
Integrating the above equation with respect to $x$, we obtain
\begin{equation}\label{eq_DsSx}
\begin{split}
    D_S (S_x)^2&=S(x)^2-S_0^2-\int_0^x\frac{2C_0 S_x
(\xi)\mathrm{d}\xi}{\alpha v(S(\xi))^2+v(S(\xi))}\\
    &=f(S(x)),
\end{split}
\end{equation}
where
\begin{equation}\label{eq_fS}
f(S)=S^2-S_0^2-\int_{S_0}^S\frac{2C_0\mathrm{d}s}{\alpha v(s)^2+v(s)}.
\end{equation}
Notice that 
\begin{equation*}
f'(S)=2S-\frac{2C_0}{v(S)(\alpha v(S)+1)}
=\frac{2}{v(S)} \left[ Sv(S)-\frac{C_0}{\alpha v(S)+1} \right],
\end{equation*}
and thus $f(S)$ is positive near $S_0$, $S>S_0$ because, according to \eqref{Cond0} and \eqref{eq_C0},
$$
f(S_0)=0, \qquad f'(S_0)=2(S_0-\rho_0)>0.
$$
Also, to have a non monotone solution for $x>0$, $S_x$ should vanish and thus there should exist $S_L$ satisfying
\[
S_L>S_0, \qquad f(S_L)=0, \qquad f(S) >0 \text { for } S_0<S<S_L .
\]
If such an $S_L$ exists, there exists a one-dimensional periodic solution of half-period $L$ and mass $M=\int_0^L\rho(x)
\mathrm{d}x$. These parameters $L$ and $M$ are determined by
\begin{equation}\label{eqL}
L=\sqrt{D_S}\int_{S_0}^{S_L}\frac{\mathrm{d}S}{\sqrt{f(S)}},
\qquad \quad M=\sqrt{D_S}
\int_{S_0}^{S_L}\frac{S\mathrm{d}S}{\sqrt{f(S)}} ,
\end{equation}
which are obtained by the following direct computations:
\begin{equation}\label{eqLM}
L=\int_0^L \mathrm{d} x=\int_{S_0}^{S_L}\frac{\mathrm{d}S}{S_x}=\sqrt{D_S}\int_{S_0}^{S_L}\frac{\mathrm{d}S}{\sqrt{f(S)}}, \qquad 
M=\int_0^LS \mathrm{d} x
=\sqrt{D_S}\int_{S_0}^{S_L}\frac{S \mathrm{d}S}{\sqrt{f(S)}}.
\end{equation}
Another question is to find such a solution for given $L$ and $M$, keeping in mind that $\rho_0$ and $S_0$ are our choices. It is difficult to solve this inverse problem. 

\blue{We can prove the existence of $S_L$ as follows. Under the assumption~\eqref{as_per_v}, $f'(S)\ll -1$ at $S\approx +\infty$ and thus, there exists $S_*>S_0$ and $S_L>S_*$ such that
$$
f'(S_*)=0,\qquad f'(S)>0\quad \mathrm{for} \quad S<S_*,\quad \text{and} \quad  f(S_L)=0.
$$
Furthermore, under the additional assumption 
\begin{equation} \label{as:per_v2}
(Sv(S))'= v(S)+Sv'(S)\leq 0 \qquad  \text{for} \quad  S>S_*, 
\end{equation}
we can estimate for $S>S_*$
}
$$
f''(S)=2+\frac{2C_0v'(S)(2\al v(S)+1)}{v(S)^2(\al v(S)+1)^2} < 2-\frac{2C_0}{Sv(S)(\al v(S)+1)}<0.
$$
Thus $f$ is concave and the positive root $S_L$ is unique.
\\

Since $f(S_0)=f(S_L)=0$ and $f(S)>0$ for $S_0<S<S_L$, the spatial derivative of $S(x)$  is given, from Eq.~(\ref{eq_DsSx}), as
\begin{equation}\label{eq_Sx+}
S_x(x)=+\sqrt{\frac{f(S(x))}{D_S}}, \quad \text{for}  \quad 0\le x\le L, \qquad S(0)=S_0, 
\end{equation}
where $L$ is defined to satisfy $S(L)=S_L$ (or \red{equivalently}, $S_x(L)=0$) according to Eq.~\eqref{eqLM}.
Thus, on the one hand,  the monotone profile of $S(x)$ for $0\le x \le L$ is obtained by Eq.~(\ref{eq_Sx+}).
On the other hand, considering the negative sign of Eq.~(\ref{eq_Sx+}), we can construct the symmetric solution $S(x)$ for $-L\le x\le 0$ as $S(x)=S(-x)$. The construction can be continued and there are periodic solutions of period $2L$, $4L$, \dots
\end{proof}

To conclude, we present the formulas, which are used for our numerical scheme in Eq.~(\ref{eqxj}):
\[
L=\sqrt{D_S}\int_{S_0}^{S_L}\frac{\mathrm{d}S}{\sqrt{f(S)}}
=2\sqrt{D_S}\int_{S_0}^{S_L}\frac{\big(\sqrt{f(S)} \big)'}{f'(S)} \mathrm{d}S=2\sqrt{D_S}\int_{S_0}^{S_L}\sqrt{f(S)}\frac{f''(S)}{(f'(S))^2}\mathrm{d}S.
\]

\subsection{Dirac concentration in periodic patterns as $D_S \to 0$}
\label{sec:concentration}

\blue{As depicted in the numerical examples on Fig.~\ref{fig:rho_Ds} and~\ref{fig:S_Ds}, the steady state solution can exhibit very localized patterns. To explain this, we now work in an interval $[0,1]$ and analyze if concentration can occur, say at $x=\frac 12$, for the solution built in Section~\ref{sec:periodic}. 
}

\blue{
Because $S$ satisfies an elliptic equation, a singular pattern can only occur as $D_S \to 0$. 
In this circumstance, it follows from Eq.~\eqref{eqS} that $S\approx \rho$ and thus the  Dirac concentration occurs both for $\rho$ and~$S$.
Consequently, far from the concentration point, i.e., as $x=0$,  we have to choose $\rho_0,\; S_0 \approx 0$ and thus $C_0\approx 0$ because of the definition~\eqref{eq_C0}. 
Injecting $C_0\approx 0$ in~\eqref{eq_fS}, the root $S_L$, if it exists, is tending to $+\infty$ since the  derivative of $f$ is positive for moderate $S$ and the second root $S_L$ exists  under the assumption 
}
\begin{equation} \label{condDirac}
 \lim_{S\to \infty} S- \frac{ C_0}{v(S)(\al v(S)+1)} < 0.
\end{equation}

\blue{We may connect this condition to the numerical illustrations below.
}
In the example \eqref{eqvS},  this means 
\[
\lim_{S\to \infty} S-  C_0 (1+qS^p) <0
\]
and concentration may happen for $p>1$. 

However for the choice \eqref{eq_vS_atan}, $v(S)$ has a positive limit as $S\to \infty$. Consequently, we find that the concentration condition~\eqref{condDirac} cannot hold. Numerically, \red{small-amplitude sinusoidal oscillations, which are consistent with the linearized equation (\ref{eq_linearized}), and phase-separation-like profiles} are observed, see Figs.~\ref{fig:tevol_rho_ent_atan} and \ref{fig:tevol_rho_ent_atan_2}.
In general, these periodic patterns are expected when the linear instability condition~\eqref{eqInsta} is fulfilled and the concentration condition~\eqref{condDirac} does not hold.

\section{Numerical illustration}
\label{sec:numerics}
We consider Eq.~\eqref{Basic} coupled with Eq.~\eqref{eqS} in the space $x\in[-L/2,L/2]$ with the periodic boundary condition.
We rewrite Eq.~(\ref{Basic2}) as
\begin{equation*}
    \partial_t \rho -D\partial_x(\rho \phi[\rho])=0,
\end{equation*}
with
\begin{equation*}
    \phi[\rho]= v(S)^2\partial_x \ln[\rho(\al v(S)^2+v(S))].
\end{equation*}
By utilizing this formula, a finite volume scheme of Eq.~(\ref{Basic}) in the one-dimensional periodic case is obtained as
\begin{subequations}\label{1dfv}
\begin{align}
    &\rho_i^{n+1}=\rho_i^n+\frac{D \varDelta t}{\varDelta x}
    (F_i-F_{i-1}),\\
    &F_i=(\phi_i^n)^+\rho_{i-1}-(-\phi_i^n)^+\rho_i,\\
    &\phi_i^n=\frac{1}{\varDelta x}v\left(\frac{S_{i-1}+S_i}{2}\right)^2\ln\left[\frac{\rho_i(\al v(S_i)^2+v(S_i))}{\rho_{i-1}(\al v(S_{i-1})^2+v(S_{i-1}))}\right],
\end{align}
\end{subequations}
where $$\rho^n_i=\frac1{\varDelta x}\int_{x_i}^{x_{i+1}}\rho(x,t^n)
\,\mathrm{d}x,\quad S^n_i=\frac1{\varDelta x}\int_{x_i}^{x_{i+1}}S(x,t^n)
\,\mathrm{d}x,$$ and $(a)^+=\max(a,0)$.
Here, the one-dimensional space $x\in[-L/2,L/2]$ is discretized as $x_i=-\frac{L}{2}+i\varDelta x$ ($i=0,\cdots,I$) with $\varDelta x=L/I$.
The periodic boundary condition is given by setting $\rho_{-1}=\rho_{I-1}$, $\rho_I=\rho_0$, $S_{-1}=S_{I-1}$, and $S_I=S_0$ in Eq.~(\ref{1dfv}).
In the following numerical simulations, we set $\varDelta x=\sqrt{D_S}/10$ and $\varDelta t=(\varDelta x)^2/5$, and fix the parameters $D=1$, $L=1$.

For the equilibrium velocity $v(S)$, we first consider
\begin{equation}\label{eqvS}
v(S)=\frac{1}{1+q S^p},
\end{equation}
with $p>1$ and $q>0$,
an expression satisfying the condition \eqref{entrpieCond} leading to entropy control.

Figures~\ref{fig:rho_al}--\ref{fig:rho_Ds} show the time evolution of the density $\rho$, starting from the same initial distribution indicated by the dashed line in each figure.
Here, we fix $p=q=2$ in Eq.~(\ref{eqvS}).

Figure~\ref{fig:rho_al} shows the effect of the parameter $\alpha=m  \lambda$, which is introduced in Sec.~\ref{sec:alpha} to connect two different limiting cases, i.e., the strong friction limit ($\alpha=0$) and the fast tumbling limit ($\alpha=\infty$).
In each figure, $k_\mathrm{c}$ denotes the critical wave number for the linear instability Eq.~(\ref{eqInsta}), i.e.,
\[
k_\mathrm{c}=\sqrt{-\frac{\gamma_\alpha +1}{D_S}}.
\]
Only oscillations with the lower wave number $k<k_\mathrm{c}$ (i.e., the longer wave lengths $l>2\pi/k_\mathrm{c}$) are linearly unstable.
Since the periodic length $L=1$ is fixed, oscillations appear only when $2\pi/k_\mathrm{c}\le 1$.

The numerical results in Fig.~\ref{fig:rho_al} illustrate that the instability condition (\ref{eqInsta}) is sharp.
It is observed that when tumbling is dominant ($\alpha=100$), aggregation is enhanced.
On the other hand, when friction is dominant ($\alpha=0$), aggregation is weakened.
This result is consistent with the linear stability analysis.
In the analysis, the parameter $\gamma_\alpha$, Eq.~(\ref{eq_gamma}), is a decreasing function of $\alpha$, so that oscillations more likely occur for large values of $\alpha$.

Figure~\ref{fig:rho_Ds} shows the numerical results obtained at different values of $D_S$.
Here, the parameter $\alpha=1$ is fixed.
Interestingly, the transient behaviors of oscillation modes are observed.
For $D_S=0.01$ [in (a)], the critical wave length $2\pi/k_\mathrm{c}=0.770$ is larger than $L/2$. 
Thus, an oscillation mode with multiple peaks is not allowed.
On the other hand, for $D_S=0.0025$ [in (b)] and $D_S=0.000625$ [in (c)], the critical wave lengths are shorter than half of the interval $L$.
Thus, oscillation modes with multiple peaks may appear due to the linear instability.
Indeed, at the initial stage, oscillation modes with multiple peaks are observed for $D_S=0.0025$ and $D_S=0.000625$.
However, those oscillation modes are merged into a unimodal aggregation at the steady state.
This result illustrates the transition to the Dirac concentration as $D_S\rightarrow 0$, which is described in Sec.~\ref{sec:concentration}.

Figure \ref{fig:S_Ds} compares the steady state solutions of $S$ obtained by the finite volume scheme (\ref{1dfv}) and by Eq.~(\ref{eq_DsSx}).
Here, Eq~(\ref{eq_DsSx}) is numerically solved as following:
When $\rho_0$ and $S_0$ are given,
$f(S)$ for $v(S)$ defined by Eq.~(\ref{eqvS}) with $p=2$ is calculated as
\[
\begin{split}
f(S)=S^2-S_0^2-2C_0&\left[(1-\alpha)(S-S_0)
+\frac{q}{3}(S^3-S_0^3) \right.\\
&\left.
+\frac{\alpha^2}{q}
\sqrt{\frac{q}{1+\alpha}}\left(
\arctan\left(\sqrt{\frac{q}{1+\alpha}}S\right)
-\arctan\left(\sqrt{\frac{q}{1+\alpha}}S_0\right)
\right)
\right].
\end{split}
\]
We uniformly descritize $S$ as $S_j=S_0+j\varDelta S$ ($j\ge 0$) and calculate the position $x_j$, at which $S(x_j)=S_j$ holds, by
\[
x_j=\sqrt{D_S}\int_{S_0}^{S_j}\frac{dS}{\sqrt{f(S)}}.
\]
We can approximate the above integration by the trapezoidal rule.
However, to avoid the numerical singularity due to $f(S_0)=0$, we use the following formulas at $j=1$,
\[
\int_{S_0}^{S_1}\frac{\mathrm{d}S}{\sqrt{f(S)}}=
\frac{2\sqrt{f(S_1)}}{f'(S_1)}+\int_{S_0}^{S_1}\frac{2f''(S)\sqrt{f(S)}}{f'(S)^2}\mathrm{d}S.
\]
Then, we can calculate $x_j$ ($j=0,\cdots,J)$ as
\begin{subequations}\label{eqxj}
\begin{align}
&x_0=0,\\
&x_1=\sqrt{D_S}\left(\frac{2\sqrt{f(S_1)}}{f'(S_1)}+\frac{f''(S_1)\sqrt{f(S_1)}}{f'(S_1)^2}\varDelta S\right),\\
&x_j=x_{j-1}+\sqrt{D_S}\frac{\varDelta S}{2}
\left(
\frac{1}{\sqrt{f(S_j)}}
+\frac{1}{\sqrt{f(S_{j-1})}}
\right),\quad(2\le j\le J).
\end{align}
\end{subequations}
Here, $J$ is determined by the condition $f(S_J)\ge 0$ and $f(S_{J+1})< 0$.

The semi-analytical results in Fig.~\ref{fig:S_Ds} are calculated using $\rho_0$ and $S_0$ of the finite volume results and then shifted along the $x$-axis to align their $S$ peak with that of the finite volume results.
It is clearly seen that the steady state solutions obtained by the finite volume scheme agree well with the semi-analytical ones.

\begin{figure}[htbp]
    \centering
    \includegraphics[width=1\linewidth]{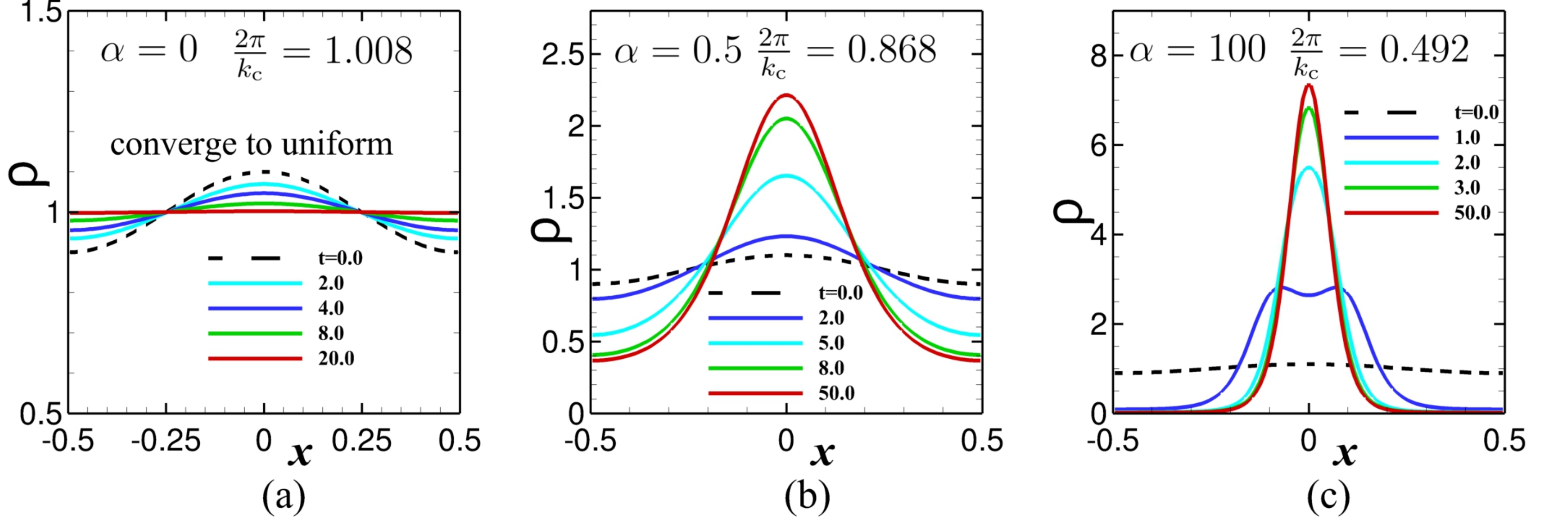}
    \caption{Time evolutions of $\rho$ at different values of $\alpha$. The parameter $D_S=0.01$ is fixed.}
    \label{fig:rho_al}
\end{figure}

\begin{figure}[htbp]
    \centering
    \includegraphics[width=1\linewidth]{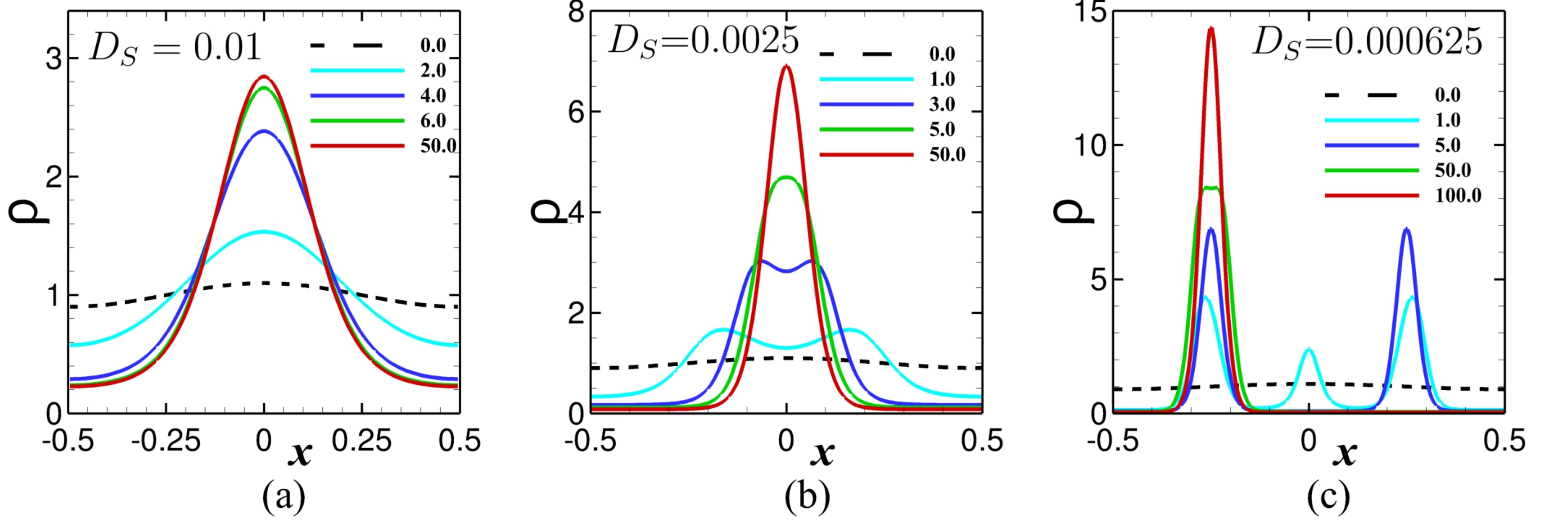}
    \caption{Time evolutions of $\rho$ at different values of $D_S$.
    The parameter $\alpha=1$ is fixed.
    The critial wave lengths are $2\pi/k_\mathrm{c}=0.770$ for (a),
    $2\pi/k_\mathrm{c}=0.385$ for (b), and
    $2\pi/k_\mathrm{c}=0.192$ for (c)
    }
    \label{fig:rho_Ds}
\end{figure}

\begin{figure}[htbp]
    \centering
    \includegraphics[width=1\linewidth]{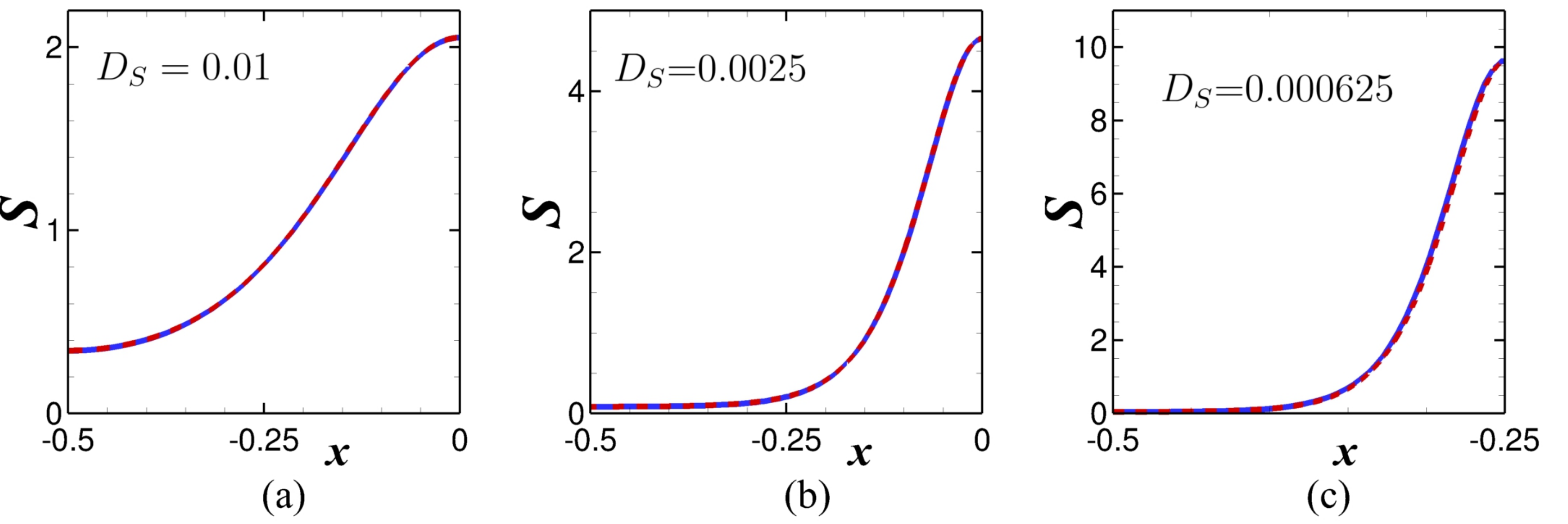} 
      \caption{Comparison of the steady solutions of $S$ obtained by the finite volume scheme (\ref{1dfv}) [blue solid lines] and those obtained by the semi-analytical method (\ref{eq_DsSx}) [red dashed lines].
    The parameters are the same as in Fig.~\ref{fig:rho_Ds}.}
    \label{fig:S_Ds}
\end{figure}

Next, we consider a sigmoidal equilibrium velocity, i.e.,
\begin{equation}\label{eq_vS_atan}
    v(S)=1-\chi \tan^{-1} \left(\frac{S-1}{\delta}\right),
\end{equation}
where $\chi>0$ is the modulation amplitude and $\delta>0$ is the stiffness.
This $v(S)$ has a positive limit as $S\rightarrow\infty$, so that we can expect \red{different types of periodic patterns} instead of a unimodal aggregation, even at small values of $D_S$ (see Sec.~\ref {sec:concentration}).

Figure \ref{fig:tevol_rho_ent_atan} shows the time evolution of density $\rho$ and entropy $E=\frac{1}{L}\int \rho\ln\rho \,\mathrm{d}x$ for $v(S)$ given by Eq.~(\ref{eq_vS_atan}) with $\chi=0.02$ and $\delta=0.01$.
Here, the simulation starts from a small periodic perturbation whose wavelength is close to the critical wavelength $2\pi/k_\mathrm{c}$.
Interestingly, multiple transitions of meta-stable states are observed.
The initial meta-stable state is a sinusoidal-like periodic pattern with a wavelength of $L/3$.
Then, it transits to a rectangular-like periodic pattern with a wavelength of $L/2$, and finally, a phase-separation-like profile is obtained.
The mode transitions in Figure (a) occur at the sharp decreases of the entropy in Figure (b).
The nonlinear analysis of this mode transition remains open.

Figure \ref{fig:tevol_rho_ent_atan_2} is the results for Eq.~(\ref{eq_vS_atan}) with $\chi=0.012$ and $\delta=0.01$.
The oscillation amplitude and entropy only grow in an early state, say $t\lesssim 1$, but they remain unchanged over the long time period $1<t<100$.
The mode transition observed in Fig.~\ref{fig:tevol_rho_ent_atan} does not occur for the long time period.
These numerical results indicate that there exist three distinct types of patterns, i.e, periodic oscillations, phase separations, and single Dirac concentrations.

\begin{figure}[htbp]
    \centering
    \includegraphics[width=0.8\linewidth]{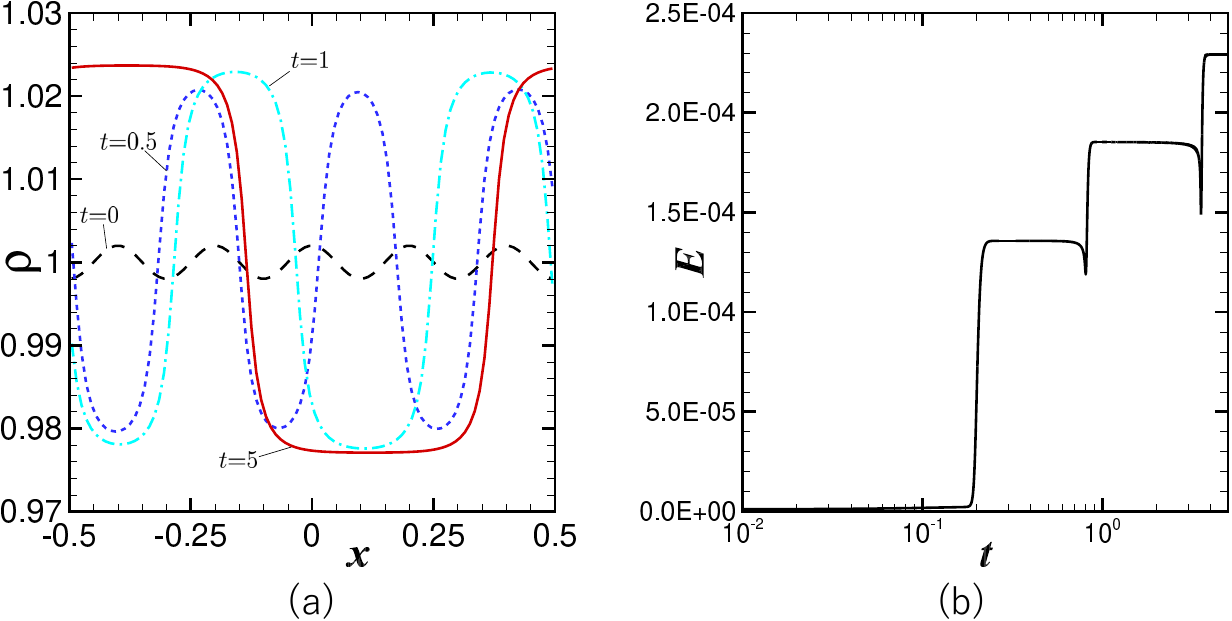}
    \caption{Time evolutions of $\rho$ [in (a)] and $E=\frac{1}{L}\int \rho\ln\rho \,\mathrm{d}x$ [in (b)] for $v(S)$ written as Eq.~(\ref{eq_vS_atan}). The parameters are set as $\chi$=0.02, $\delta$=0.01, $D_S$=0.001, and $\alpha=0$, which gives $2\pi/k_{\mathrm{c}}=0.20$. }
    \label{fig:tevol_rho_ent_atan}
\end{figure}

\begin{figure}[htbp]
    \centering
    \includegraphics[width=0.8\linewidth]{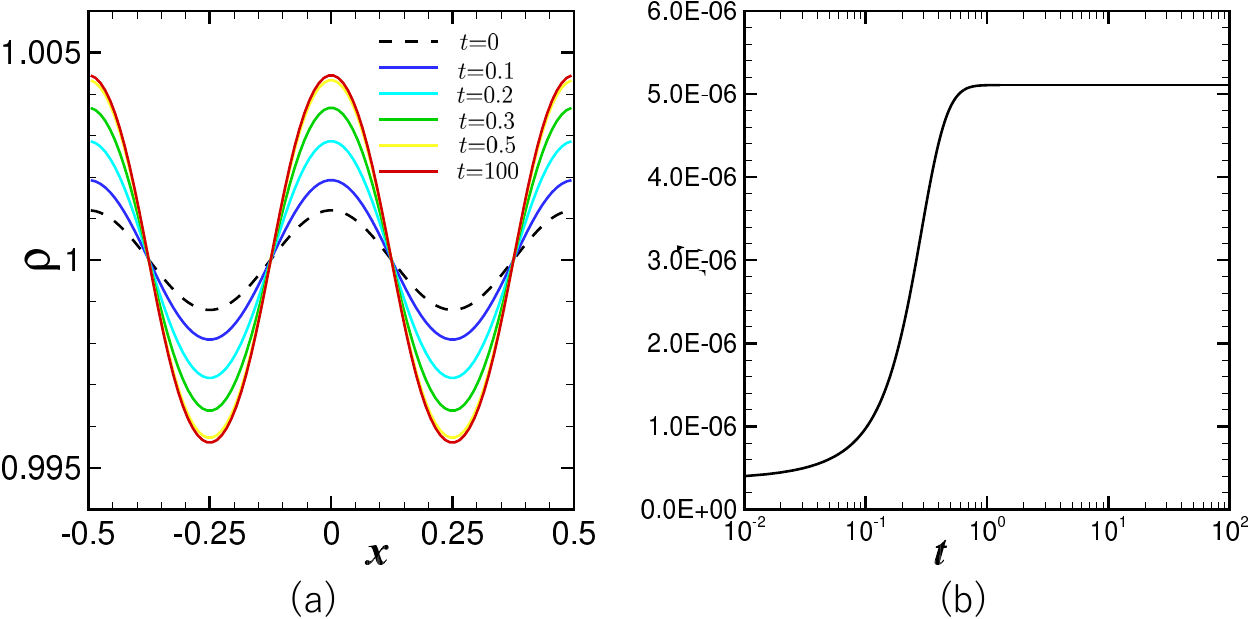}
    \caption{Time evolutions of $\rho$ [in (a)] and $E=\frac{1}{L}\int \rho\ln\rho \,\mathrm{d}x$ [in (b)] for $v(S)$ written as Eq.~(\ref{eq_vS_atan}). The parameters are set as $\chi$=0.012, $\delta$=0.01, $D_S$=0.001, and $\alpha=0$, which gives $2\pi/k_{\mathrm{c}}=0.44$. }
    \label{fig:tevol_rho_ent_atan_2}
\end{figure}

\section{Conclusion}

The new kinetic model \eqref{eq:K1} takes into account the interaction of cells with the substrate through some mechanical forces (propulsion and friction). \blue{In this way, it provides a mesoscopic framework for mechanotaxis, extending standard run-and-tumble models by explicitly incorporating the effect of the mechanical environment on cell migration and collective behavior.}

The relative values of tumbling rate and cell mass determine a family of macroscopic equations. \blue{More precisely, starting from the kinetic model, we formally derived several macroscopic limit equations corresponding to different asymptotic regimes.} These are Fokker--Planck equations with similarity to the Keller--Segel system on the one hand, and with density-suppressed diffusion on the other hand. \blue{One of the main contributions of the paper is therefore to clarify how these different macroscopic behaviors emerge from a single mechanotactic kinetic description.}
\red{Thus, the macroscopic transport coefficients have a clear physical interpretation, since they are derived from the mesoscopic equation without imposing ad-hoc macroscopic assumptions.}

We derived the linear stability condition which combines a parameter value (expressing the relation between tumbling and friction) and signal dependent equilibrium velocity. This condition is very sharp as confirmed by numerical simulations.
\blue{For some of the macroscopic settings, we also obtained more detailed analytical results, including structural properties related to entropy dissipation and free-energy decay.} 

In unstable situations, we found three types of typical patterns: periodic oscillations, phase separations, and single Dirac concentrations (small signal diffusion). \blue{At the theoretical level, these behaviors are predicted by the linear and semi-analytical analysis of the macroscopic equations.} 
\blue{The numerical illustrations are restricted to one space dimension, where they validate the sharpness of the instability threshold and illustrate the possible pattern-selection mechanisms. They should however be understood as illustrative examples rather than as a complete characterization of the nonlinear multidimensional dynamics.}

Numerics shows that several metastable states may occur and understanding the transitions and the global dynamics remains open. From the theoretical side, a nonlinear analysis for the coupled system remains to be done. \blue{A fully rigorous justification of the macroscopic limits from the kinetic model is also left for future work.} In particular we expect blow-up phenomena, at least in dimension $d\geq 2$, as in the limiting cases of the pure Keller--Segel system and, in some specific situations, for the density-suppressed systems. The model itself could be expanded to treat different physical situations as deformation of the substrate and more elaborated hydrodynamic effects.
\blue{Another possible extension of the model could consist in investigating the tumbling frequency modification induced by the presence of supplementary internal variables, according, e.g., to  \cite{ErbanOth2,PTV2016,PSTY2020}.}

\section*{Acknowledgements}
SY acknowledges financial support from JSPS KAKENHI Grant Number JP25K07246.

Part of this paper has been written when FS visited the University of Hyogo under the JSPS invitational fellowship
scheme. JSPS, Kyoto University and the University of Hyogo are deeply acknowledged for the hospitality and for the financial support of the visit.
FS has been moreover partially supported by INdAM, GNFM group.

\appendix

%
%

\bibliographystyle{plain}
\bibliography{biblio}


\end{document}